\documentclass[10pt,reqno]{amsart}
\usepackage{a4wide}
\usepackage{amssymb}
\usepackage[toc,page]{appendix}
\newtheorem{thm}{Theorem}[section]
\newtheorem{cor}[thm]{Corollary}
\newtheorem{lem}[thm]{Lemma}
\theoremstyle{remark}
\newtheorem{rem}[thm]{Remark}

\theoremstyle{definition}
\newtheorem{defin}[thm]{Definition}

\DeclareMathOperator{\M}{M} 
\DeclareMathOperator{\Mor}{Mor} \DeclareMathOperator{\B}{B}
\DeclareMathOperator{\Rep}{Rep} \DeclareMathOperator{\D}{D}
\DeclareMathOperator{\C}{C} \DeclareMathOperator{\Aut}{Aut}
\DeclareMathOperator{\Op}{Op} \DeclareMathOperator{\Graph}{Graph}
\DeclareMathOperator{\Ran}{Ran}
\newcommand{\ha}{\hat\alpha}
\newcommand{\hb}{\hat\beta}
\newcommand{\hg}{\hat\gamma}
\newcommand{\hd}{\hat\delta}
\newcommand{\ta}{\tilde\alpha}
\newcommand{\tb}{\tilde\beta}
\newcommand{\tg}{\tilde\gamma}
\newcommand{\td}{\tilde\delta}
\newcommand{\gen}{\ha,\hb,\hg,\hd}
\newcommand{\Del}{\Delta}
\newcommand{\kap}{\kappa}
\newcommand{\ot}{\otimes}
\newcommand{\ep}{\varepsilon}
\newcommand{\hsp}[1]{&\hspace*{-0,25cm}#1&\hspace*{-0,25cm}}
\newcommand{\mc}[1]{\mathcal{#1}}
\newcommand{\mb}[1]{\mathbb{#1}}
\newcommand{\me}[3]{\langle\, #1\,|\,#2\,|\,#3\,\rangle}
\newcommand{\id}{{\rm id}}
\newcommand{\SL}{SL(2,\mathbb{C})}
\begin{document}
\title[]{The Heisenberg-Lorentz quantum group}
\author{P.~Kasprzak}
\address{Department of Mathematical Methods in Physics\\
Faculty of Physics\\
Warsaw University}
\email{pawel.kasprzak@fuw.edu.pl}
\thanks{The research was  supported by the KBN under grant 115/E-343/SPB/6.PR UE/DIE 50/2005 – 2008.}
\begin{abstract} In this article we  present a new $\C^*$-algebraic deformation of the Lorentz group. It is obtained by means of the Rieffel
deformation applied to $\SL$. We give a detailed description of the resulting  quantum
group $\mathbb G=(A,\Delta)$ in terms of generators
$\ha,\hb,\hg,\hd\in A^\eta$ - the quantum counterparts of the matrix coefficients $\alpha,\beta,\gamma,\delta$ of the fundamental representation of $\SL$.
In order to construct $\hb$ - the most involved  of four generators -  we first define it on the quantum Borel subgroup $\mathbb G_0\subset\mathbb G$, then on the quantum complement of the Borel subgroup and finally we perform the gluing procedure.
 In order to classify representations of the
$\C^*$-algebra $A$ and to analyze the action of the comultiplication $\Delta$ on the generators $\ha,\hb,\hg,\hd$ we employ the duality in the theory of locally compact quantum groups.
\end{abstract}
\maketitle \tableofcontents
\begin{section}{Introduction} A complete classification of
deformations of $SL(2,\mb{C})$ on the Hopf $*$-algebra level was
presented in \cite{WorZak}. Up till now three of the cases contained
there have been realized on a deeper, $\C^*$-algebraic level (see
\cite{Kasp}, \cite{PodWor}, \cite{WorZak1}). This paper is devoted
to the $\C^*$-algebraic realization of another case. The method of
deformation that we use is the Rieffel deformation which is the
same as in the example considered in \cite{Kasp}. Nevertheless, the
resulting quantum group $\mb{G}=(A,\Delta)$ - the Heisenberg-Lorentz quantum group -  is much more complex. One of the difficulties lies in the fact that
among the four generators $\ha, \hb,\hg,\hd$ of the $\C^*$-algebra
$A$, only $\hg$ is normal. Also the analysis of the
comultiplication $\Delta$ is not as straightforward as in the case
of \cite{Kasp}. To perform it we use the one to one correspondence
between representations of $\C^*$-algebra $A$ and
corepresentations of the dual quantum group $\widehat{\mb{G}}$. This
correspondence was also used to describe all representations of
$A$ on Hilbert spaces.

Let us briefly describe the contents of the paper. In the next
section we present the Hopf $*$-algebraic version of the
Heisenberg-Lorentz quantum group. We begin with a description of
commutation relations satisfied by generators $\ha,\hb,\hg,\hd$ - the quantum counterparts of the matrix coefficients $\alpha,\beta,\gamma,\delta$ of the fundamental representation of $\SL$.
Formulas for the comultiplication, coinverse and counit on
generators are the same as in the classical case. In  Section
\ref{hsrep} we define Hilbert space representations of the
Heisenberg-Lorentz commutation relations. We show that the tensor
product of two such representations can be defined. Section
\ref{calglev} is devoted to the construction of the
$\C^*$-algebraic version $\mathbb G=(A,\Delta)$ of the Heisenberg-Lorentz
quantum group. In particular we introduce four affiliated elements
$\ha,\hb,\hg,\hd\in A^\eta$. In order to construct $\hb$,		  we first define it on the quantum Borel subgroup $\mathbb G_0\subset\mathbb G$, then on the quantum complement of the Borel subgroup and finally we perform the gluing procedure. Having constructed affiliated elements
$\ha,\hb,\hg,\hd$, we show that they generate $A$. Moreover, we
note that for any representation $\pi\in\Rep(A;\mc{H})$ the
quadruple $(\pi(\ha),\pi(\hb),\pi(\hg),\pi(\hd))$ is a Hilbert
space representation of the Heisenberg-Lorentz commutation
relations. The converse is also true: for any representation
$(\ta,\tb,\tg,\td)$ of the Heisenberg-Lorentz commutation
relations on a Hilbert space $\mc{H}$ there exists a unique
representation $\pi\in\Rep(A;\mc{H})$ such that
\[\pi(\ha)=\ta,\,\,\,\pi(\hb)=\tb,\,\,\, \pi(\hg)=\tg,\,\,\, \pi(\hd)=\td.\]
At the end of  Section \ref{calglev} we show that the action of
$\Delta$ on generators has the same form as in the classical case.
Appendices gather useful facts concerning  the quantization map
and the counit in the Rieffel deformation, the complex
infinitesimal generator of the Heisenberg group and the product of
strongly commuting affiliated elements.

Throughout the paper we will freely use the language of
$\C^*$-algebras and the theory of locally compact quantum groups.
For a locally compact space $X$, $\C_0(X)$ and $\C_{\rm b}(X)$ shall respectively denote the
algebra of continuous functions vanishing at infinity and the algebra of
continuous bounded functions. If
$X$ is also a manifold, then $\C^\infty(X)$ denotes the algebra of
smooth functions on $X$ and $\C^\infty_c(X)$ denotes the algebra
of smooth functions of compact supports. For the notion of
multipliers, affiliated elements and algebras generated by a
family of affiliated elements we refer the reader to
\cite{WorNap}, \cite{Worgen} and \cite{Worun}. The set of elements
affiliated with a $\C^*$-algebra $A$ will be denoted by $A^\eta$ and the affiliation relation will be denoted by $\eta$, i.e.  $T\,\eta\, A$ means that $T\in A^\eta$.
The $z$-transform of $T\in A^\eta$ will be denoted by $z_T$. For
the precise definition of $z_T$ we refer to \cite{WorNap}. For the
theory of locally compact quantum groups we refer to \cite{KV} and
\cite{MNW}. For the theory of quantum groups given by a
multiplicative unitary we refer to \cite{BS} and \cite{SolWor}.
For the notion of $\Gamma$-product we refer to \cite{Ped}. All
Hilbert spaces appearing in the paper are assumed to be separable.
Given a pair of densely defined operators $X$ and $Y$ acting on a Hilbert space $\mathcal{H}$,
the dotted sum  $X\dotplus Y$ is the closure of the usual sum $X+Y$. To define $X\dotplus Y$ one has to prove that
the intersection of domains $\D(X)\cap D(Y)$ is dense in $\mathcal{H}$ and $X+Y$ defined on $\D(X)\cap\D(Y)$ is closable.

I would like to express my gratitude to S.L. Woronowicz for many stimulating discussions, which greatly influenced the final form of this paper. 
\end{section}
\begin{section}{Hopf *-algebra level}
We fix a deformation parameter $s\in\mathbb{R}$. Let $\mathcal{A}$
be a unital $*$-algebra generated by four elements $\gen$
satisfying the following commutation relations:
\begin{equation}\label{hlrel1}\left.\begin{array}{c}
\begin{array}{ccc} \ha\hb=\hb\ha&\ha\hg=\hg\ha&\ha\hd=\hd\ha\\
                                &\hb\hg=\hg\hb&\hb\hd=\hd\hb\\
                                &             &\hg\hd=\hd\hg\\
                                &&\ha\hd-\hb\hg=1
\end{array}\\\\
\begin{array}{cccc}\ha\ha^*-\ha^*\ha=-s\hg^*\hg&
\ha\hb^*-\hb^*\ha=-s\hg\hd^*&
 \ha\hg^*=\hg^*\ha&
 \ha\hd^*=\hd^*\ha\\
 &\hb\hb^*-\hb^*\hb=s(\ha^*\ha-\hd\hd^*)&\hb\hg^*=\hg^*\hb&\hb\hd^*-\hd^*\hb=s\hg^*\ha\\
 &&\hg\hg^*=\hg^*\hg&\hg\hd^*=\hd^*\hg\\
 &&&\hd\hd^*-\hd^*\hd=s\hg\hg^*.
\end{array}
\end{array}\right\}
\end{equation}
The $*$-algebra $\mathcal{A}$ was introduced in \cite{WorZak}
where it was also proven that it admits the structure of a Hopf
$*$-algebra. The action of the comultiplication $\Del:\mathcal{A}\rightarrow \mathcal{A}\otimes \mathcal{A}$ on the
generators is given by:
\begin{equation}\label{coalg}
\begin{array}{rcllrcl}
   \Del(\ha)\hsp{=}\ha\otimes\ha+\hb\otimes\hg\hsp{,}
    \Del(\hb)\hsp{=}\ha\otimes\hb+\hb\otimes\hd \\
    \Del(\hg)\hsp{=}\hg\otimes\ha+\hd\otimes\hg\hsp{,}
    \Del(\hd)\hsp{=}\hg\otimes\hb+\hd\otimes\hd.
   \end{array}
\end{equation}
The coinverse $\kap:\mathcal{A}\rightarrow \mathcal{A}$ is an involutive $*$-antihomomorphism and
its action on the generators is given by:
\begin{equation}\label{coalg1}
\begin{array}{rcllrcl}
   \kap(\ha)\hsp{=}\hd\hsp{,}\kap(\hb)\hsp{=}-\hb \\
    \kap(\hg)\hsp{=}-\hg\hsp{,}\kap(\hd)\hsp{=}\ha.
   \end{array}
\end{equation}
Finally, the action of the counit $\ep:\mathcal{A}\rightarrow\mathbb{C}$ on the generators is as follows:
\begin{equation}\label{coalg2}
\begin{array}{rcllrcl}
   \ep(\ha)\hsp{=}1\hsp{,}\ep(\hb)\hsp{=}0 \\
    \ep(\hg)\hsp{=}0\hsp{,}\ep(\hd)\hsp{=}1.
   \end{array}
\end{equation}
Note that the formulas defining co-operations on $\mathcal{A}$
coincide with the corresponding formulas for the Hopf $*$-algebra of
polynomial functions on $\SL$.
\end{section}
\begin{section}{Hilbert space level}\label{hsrep}
In this section we distinguish a class of
representations of commutation relations \eqref{hlrel1} on a
Hilbert space which will be proven to correspond to representations
of the $\C^*$-algebra $A$ of the Heisenberg-Lorentz quantum group (see Theorem \ref{repth}).
Note first that the pair $(\ha, -s\hg^*\hg)$, satisfy the commutation relation defining the Heisenberg Lie algebra
(see Appendix \ref{bcg}). The same is true of the pair $(\hd,s\hg^*\hg)$.
Furthermore, in the case when $\hg$ is represented by an invertible operator, the equation $\ha\hd-\hb\hg=1$ determines $\hb$.
This gives a motivation for the following definition.
\begin{defin}\label{reptrip} Let $\ta,\tg,\td$ be closed operators acting on a
Hilbert space $\mc{H}$. We say that the triple $(\ta,\tg,\td)$
satisfies the Heisenberg-Lorentz commutation relations if:
\begin{itemize}
\item[1.] $\tg$ is normal and $\ker\tg=\{0\}$;
\item[2.] $(\ta,-s\tg^*\tg)$ and  $(\td,s\tg^*\tg)$ are infinitesimal representations of
the Heisenberg group $\mb{H}$;
\item[3.] operators $\ta,\tg$ and $\td$ mutually strongly
commute.
\end{itemize}
\end{defin} For the notion of infinitesimal representation of $\mathbb{H}$ we refer to
Definition \ref{c0} and for the notion of strong commutativity we refer to Definition \ref{c1}.

Definition \ref{reptrip} describes representations of commutation
relations \eqref{hlrel1} in which $\hg$ is represented by an
invertible operator $\tg$. The next definition deals with the representations for which
$\tg=0$. Note that in this case it is the pair $(\tb,s(\ta^*\ta-1/\ta^*\ta))$ that satisfies the Heisenberg Lie algebra relation.
\begin{defin}\label{reppair} Let $\ta,\tb$ be closed operators acting on a
Hilbert space $\mc{H}$. We say that the pair $(\ta,\tb)$ satisfies
the Heisenberg-Lorentz commutation relations if
\begin{itemize}
\item[1.] $\ta$ is normal and $\ker\ta=\{0\}$;
\item[2.] $(\tb,s(\ta^*\ta-1/\ta^*\ta))$ is an infinitesimal representation of  of
the Heisenberg group $\mb{H}$;
\item[3.] operators $\ta$ and  $\tb$ strongly commute.
\end{itemize}
\end{defin}

To deal with the general case of representations of Heisenberg-Lorentz commutation relations note that $\hg$ commutes with all of the generators and their adjoints. This fact leads to the idea that
 any representation of the Heisenberg-Lorentz commutation relations splits into a direct sum of two representations: one with an invertible $\tg$ and one with $\tg$ being zero. More precisely we have:
\begin{defin}\label{repquar} Let $\ta,\tb,\tg,\td$ be closed
operators acting on a Hilbert space $\mc{H}$, $\tg$ being normal.
By $\mc{H}_0$, $\mc{H}_1$ and $\tg_1$ we shall respectively denote
the kernel of $\tg$, its orthogonal complement and the restriction
of $\tg$ to $\mc{H}_1$. We say that the quadruple
$(\ta,\tb,\tg,\td)$ is a representation of the Heisenberg-Lorentz
commutation relations if $\ta,\tb$ and $\td$ respect the
decomposition $\mc{H}=\mc{H}_0\oplus \mc{H}_1$ i.e. there exist
closed operators $\ta_0,\tb_0,\td_0$ acting on $\mc{H}_0$ and
$\ta_1,\tb_1,\td_1$ acting on $\mc{H}_1$ such that
\[\ta=\ta_0\oplus\ta_1,\,\, \tb=\tb_0\oplus\tb_1,\,\,
\td=\td_0\oplus\td_1\] and we have
\begin{itemize}
\item[1.] the pair $(\ta_0,\tb_0)$ satisfies the Heisenberg-Lorentz
commutation relations;
\item[2.] $\ta_0$ and $\td_0$ are mutual inverses:
$\td_0={\ta_0}^{-1}$;
\item[3.] the triple $(\ta_1,\tg_1,\td_1)$ satisfies the Heisenberg-Lorentz
commutation relations;
\item[4.] $\tb_1={\tg_1}^{-1}(\ta_1\td_1-1)$.
\end{itemize}
\end{defin}
\begin{rem}\label{dirint} The product of
operators in point 4 above is taken in the sense of Theorem
\ref{thmc2}. It is a well known fact that a representation of the
Heisenberg group $\mb{H}$ can be decomposed into a direct integral
of irreducible representations. In the case of irreducible
representations the operator $\tg$ appearing in Definition
\ref{reptrip} and $\ta$ appearing in Definition \ref{reppair}, are
multiples of identity. This fact will be used in the proof of the next theorem.
\end{rem}
As has already been mentioned, the class of representations defined above  corresponds
 to representations of the $\C^*$-algebra  $A$ of the Heisenberg-Lorentz quantum group (see Theorem \ref{repth}). Let us recall that for two representations $\pi_1\in\Rep(A,\mc{H})$ and $\pi_2\in\Rep(A,\mc{H}')$ their tensor product is defined by $\pi=(\pi_1\otimes\pi_2)\circ\Delta\in\Rep(A;\mc{H}\otimes \mc{H}')$. The next theorem gives a description of the tensor product construction in terms of the Heisenberg-Lorentz commutation relations. This construction will be crucial in the analysis of the comultiplication $\Delta$ on the $\C^*$-algebra level (see Theorem \eqref{delg}).
\begin{thm}\label{comrep}
Let $\ta,\tb,\tg,\td$ be closed operators acting on a Hilbert
space $\mc{H}$ and let $\ta',\tb',\tg',\td'$ be closed operators
acting on a Hilbert space $\mc{H}'$. Assume that
$(\ta,\tb,\tg,\td)$, $(\ta',\tb',\tg',\td')$ are representations
of the Heisenberg-Lorentz commutation relations. Then the
quadruple of operators $(\ta'',\tb'',\tg'',\td'')$ acting on
$\mc{H}\otimes \mc{H}'$, defined by:
\begin{equation}\label{delact}
\begin{array}{rcl}
\ta''\hsp{=}\ta\otimes\ta'\dotplus\tb\otimes\tg'\\
\tb''\hsp{=}\ta\ot\tb'\dotplus\tb\ot\td' \\
\tg''\hsp{=}\tg\ot\ta'\dotplus\td\ot\tg'\\
\td''\hsp{=}\tg\ot\tb'\dotplus\td\ot\td'
\end{array}\end{equation}
is a representation of the Heisenberg-Lorentz commutation
relations on $\mathcal{H}\otimes\mathcal{H}'$.
\end{thm}
\begin{proof} First, let us introduce some notation. For any $\varepsilon>0$,
$z\in\mathbb{C}$ we set
\[f_\varepsilon(z)=\frac{1}{\pi\ep}
\exp(-{\varepsilon}^{-1}|z|^2)\in\mb{R}_+.\] Note  that for any $\ep\in\mb{R}_+$, $f_\ep\in
L^1(\mb{C})$  and the family $f_\ep$
is a Dirac delta approximation:
\[\lim _{\ep\rightarrow 0}\int d^2z\, f_{\ep}(z)g(z)=g(0)\]
where $d^2z$ is a Haar measure on $\mathbb{C}$.
Let $V^a$ be an irreducible unitary
representation of the Heisenberg group $\mb{H}$ on a Hilbert space $\mc{H}$ (for an
explanation of the notation $V^a$ we refer to Appendix
\ref{bcg}). Smearing the family $V^a_{z,0}$ with a function $f_{\ep}$ we get the family $I^a_{\ep}$ of bounded operators acting on $\mathcal{H}$:
\begin{equation}\label{auxop}I^a_\ep=\int d^2z\,f_\varepsilon(z) V^{a}_{z,0}.\end{equation}
The following properties of $I^a_\ep$ will be used in the course of the proof:
\begin{equation}\label{iep}\left\{
\begin{array}{lll}
1.&\displaystyle \mbox{s}-\lim_{\ep\rightarrow 0}I^a_\varepsilon=1,&\hspace{5.5cm}\\[0.4ex]
2.& \displaystyle \Ran(I_\varepsilon)\subset\D(a^n)\mbox{ for any }n\in\mb{N},&\\[0.4ex]
3. &\displaystyle\lim_{\ep\rightarrow 0}a^nI_\ep h=a^nh\mbox{ for any
}h\in D(a^n),&
\end{array}\right.\end{equation}  where $\mbox{s}-\lim$ denotes the limit in
the strong topology on $\B(\mc{H})$.

Let us move on to the main part of the proof.
Let $c,c'\in\mb{C}\setminus\{0\}$. By Remark \ref{dirint}  it is
enough to prove our theorem in the following  three cases:
\begin{equation}\label{list01}\left\{
\begin{array}{lll}
1. & \mc{H}_0=\mc{H}_0'=\{0\},\,\, \tg=c1 \mbox{ and } \tg'=c'1, &\hspace{5cm}\\
2. & \mc{H}_0=\mc{H}_1'=\{0\},\,\, \tg=c1 \mbox{ and } \ta_0'=c'1,  &\\
3. & \mc{H}_1=\mc{H}_1'=\{0\},\,\, \ta_0=c1 \mbox{ and }
\ta'_0=c'1. &
\end{array}\right.
\end{equation}
The notation used above coincides with the notation of Definition
\ref{repquar}. In what follows we shall treat case 1 leaving cases
2 and 3 to the reader.

Note that the pairs $(1\ot
c\ta',-s|cc'|^2)$ and $(c'\td\ot 1, s|cc'|^2)$ are infinitesimal
representations of $\mb{H}$. For any $z\in\mb{C}$ we define a
unitary operator:
\[U_z=U^{1\ot c\ta'}_{z,0}U^{c'\td\ot
1}_{z,0}\in\B(\mc{H}_1\otimes \mc{H}'_1).\] It is easy to check
that the map
\[\mathbb{C}\ni z\mapsto
U_z\in\B(\mc{H}_1\otimes \mc{H}'_1)\] is a strongly continuous
representation of the group $(\mathbb{C},+)$. Let $T$ be the
corresponding infinitesimal generator. By definition $T$ is a
normal operator with the domain
\[\D(T)=\left\{h\in\mc{H}_1\otimes \mc{H}'_1:
\begin{array}{l}\mbox{The map}\\\mb{C}\ni z\mapsto
U_z h\in\mc{H}_1\otimes \mc{H}'_1\\\mbox{is once
differentiable}\end{array}\right\}\] and the action of $T$ on
$h\in\D(T)$ is given by
\begin{equation}\label{td}Th=2\frac{\partial}{\partial
z}U_zh\Big|_{z=0}.\end{equation}
With this definition of $T$ we have $\displaystyle U_z=e^{i\Im(zT)}$, which explains the factor
$2$ on the right hand side of \eqref{td}.
Comparing formulas \eqref{td} and \eqref{remdif} we see
that $\tg''\subset T$. In order to prove the equality $\tg''=T$ it is
enough to show that the linear subspace $\D(1\ot
\ta')\cap\D(\td\ot 1)\subset \mc{H}_1\ot \mc{H}'_1$, which is a core of $\tg''$ is also a core
of $T$. For this we use the family of operators
$I_\ep^{\td}\ot I^{\ta'}_\ep\in\B(\mc{H}_1\ot \mc{H}'_1)$. It has
the following properties:
\begin{equation}\label{list}\left\{
\begin{array}{lll}
1.&\displaystyle \mbox{s}-\lim_{\varepsilon\rightarrow 0}(I_\ep^{\td}\ot
I^{\ta'}_\ep)=1,&\hspace{5cm}\\[0.4ex]
2.&  \Ran(I_\ep^{\td}\ot I^{\ta'}_\ep)\subset
\D(\td\ot 1)\cap\D(1\ot \ta'),&\\[0.4ex]
3. &\displaystyle\lim_{\varepsilon\rightarrow 0}T(I_\ep^{\td}\ot
I^{\ta'}_\ep)h=Th\mbox{ for any }h\in D(T).&
\end{array}\right.\end{equation} Properties 1 and 2 above are direct
consequences of \eqref{iep}, while property 3 demands a separate proof
which is based on formulas \eqref{auxop} and \eqref{td}. The fact
that $\D(1\ot \ta')\cap\D(\td\ot 1)$ is a core of $T$ is now an
immediate consequence of 1, 2 and 3, hence we have $T=\tg''$.

In the analysis of $\tb''$ we shall use the fact that $\tg''$ defined above is invertible: $\ker\tg''=\{0\}$.
 Assume in the contrary that $\ker\tg''\neq\{0\}$.
Using the following identity:
\begin{equation}\label{trker} \left(U^{\td}_{\frac{z}{c'},0}\ot
U^{\ta'}_{-\frac{z}{c},0}\right)\tg''\left(U^{\td}_{-\frac{z}{c'},0}\ot
U^{\ta'}_{\frac{z}{c},0}\right)=\tg''+\bar z\end{equation}
we see that  $\tg''$  has  an eigenvector for any
complex number. This fact and the normality of $\tg''$
(eigenvectors of different eigenvalues are perpendicular)
contradicts separability of $\mc{H}_1\ot \mc{H}_1'$, hence $\ker\tg''=\{0\}$.

Let us move on to the  analysis of the operator
$\ta''=\ta\otimes\ta'\dotplus\tb\otimes\tg'$. Our objective is to
show that $\ta''$ is an infinitesimal complex generator of a
representation of the Heisenberg group $\mathbb{H}$ (cf. Definition
\ref{reptrip}). In order to do that we define an auxiliary operator
\begin{equation}\label{formal}T'=\tg''(c^{-1}\ta\ot 1)-c^{-1}c'.\end{equation}
Note that $\tg''$ and $\ta\ot 1$ strongly commute and by Theorem
\ref{thmc2}, $T'$ is well defined. It is easy to see that
$(T',-s\tg''^*\tg'')$ is an infinitesimal representation of
$\mb{H}$. Hence, to prove that $(\ta'',-s\tg''^*\tg'')$ is also an
infinitesimal representation of $\mb{H}$ it is enough to show that
$\ta''=T'$. For this purpose we use the  family of operators
$I_{\ep}\in\B(\mc{H}_1\ot \mc{H}'_1)$:
\[I_{\ep}=I^{\ta}_\ep I^{\td}_\ep\ot I^{\ta'}_\ep.\] It has the following properties:
\begin{equation}\label{list2}\left\{
\begin{array}{lll}
1.&\displaystyle \mbox{s}-\lim_{\ep\rightarrow 0}I_\varepsilon=1,&\hspace{5cm}\\[0.6ex]
2.&  \Ran(I_\ep)\subset\D(\ta'')\cap\D(T'),&\\[0.5ex]
3.& T'|_{\Ran(I_\ep)}=\ta''|_{\Ran(I_\ep)},&\vspace{0,05cm}\\[0.5ex]
4. &\displaystyle\lim_{\ep\rightarrow 0}T'I_\ep h=T'h\mbox{ for any
}h\in\D(T'),&
\\
5. &\displaystyle\lim_{\ep\rightarrow 0}\ta''I_\ep h=\ta''h\mbox{ for any
}h\in\D(\ta'').&
\end{array}\right.\end{equation}
Properties 3, 4 and 5 show that  $T'$ and $\ta''$ coincide on a
join core, hence $T'=\ta''$.

Similarly we prove that the operator
$\td''=\tg\ot\tb'\dotplus\td\ot\td'$ gives rise to the
infinitesimal representation $(\td'',-s\tg''^*\tg'')$ of the Heisenberg group $\mb{H}$.

To complete the proof we have to show that the operator
$\tb''=\ta\ot\tb'\dotplus\tb\ot\td'$ is equal to
$\tg''^{-1}(\ta''\td''-1)$ (see point 4 of Definition \ref{repquar}). In order to do that we use the family
of operators  $J_\ep\in\B(\mc{H}_1\ot \mc{H}'_1)$:
\[J_\ep=I_\ep^{\ta}I_\ep^{\td}\ot I_\ep^{\ta'}I_\ep^{\td'}.\]
It has the following properties:
\begin{equation}\label{list3}\left\{
\begin{array}{lll}
1.& \displaystyle s-\lim_{\ep\rightarrow 0}J_\varepsilon=1,&\hspace{4cm}\\[0.5ex]
2.&  \Ran(J_\ep)\subset\D(\tg''^{-1}(\ta''\td''-1))\cap\D(\tb''),&\\[0.5ex]
3.& \tb''|_{\Ran(J_\ep)}=\tg''^{-1}(\ta''\td''-1)|_{\Ran(J_\ep)}, \\[0.5ex]
4.&\displaystyle\lim_{\ep\rightarrow 0}\tb''J_\ep h=\tb''h\mbox{ for any
}h\in\D(\tb''),&
\\
5.&\displaystyle\lim_{\ep\rightarrow 0}\tg''^{-1}(\ta''\td''-1)J_\ep
h=\tg''^{-1}(\ta''\td''-1)h&\\
&\mbox{for any }h\in\D(\tg''^{-1}(\ta''\td''-1)).&
\end{array}\right.\end{equation}
Properties 3, 4 and 5 show that $\tb''$ and $
\tg''^{-1}(\ta''\td''-1)$ coincide on a join core, therefore
$\tb''=\tg''^{-1}(\ta''\td''-1)$.
\end{proof}
\end{section}
\begin{section}{$\C^*$-algebra level} \label{calglev} In this section we shall
describe the Heisenberg-Lorentz quantum group on the
$\C^*$-algebra level. It is obtained by applying the Rieffel
deformation to $\SL$ (which from now on will be denoted
by $G$). Let us fix an abelian subgroup $\Gamma\subset G$ which will be
used in the deformation procedure. We choose:
\[\Gamma=\left\{\left(
           \begin{array}{cc}
             1 & z \\
             0 & 1 \\
           \end{array}
         \right):z\in\mb{C}\right\}.\]
Note that $\Gamma$ is isomorphic with the additive group of
complex numbers. In particular $\Gamma$ and its Pontryiagin  dual
$\Hat\Gamma$ are isomorphic. The isomorphism that we shall use is
given by the following non-degenerate pairing on $\mathbb{C}\times\mathbb{C}$:
\[\langle z,z'\rangle=\exp(i\Im(zz')).\]
Let $\Psi$ be the skew bicharacter  on
$\Hat\Gamma\simeq\mathbb{C}$ given by:
\begin{equation}\label{dualbi}\Psi(z,z')=\exp\left(i\frac{s}{4}\Im(z\bar z')\right).\end{equation}
Using the results of Rieffel (see \cite{Rf2}) we know that the
abelian subgroup $\Gamma\subset G$ and the bicharacter $\Psi$ on
$\Hat\Gamma$ give rise to a quantum group $\mb{G}=(A,\Delta)$. In
this paper we shall use a formulation of the Rieffel deformation
based on the theory of crossed products which was described in
\cite{Kasp}. In this framework the deformation procedure  goes as
follows. Let $\rho:\Gamma^2\rightarrow \Aut(\C_0(G))$ be the
action of $\Gamma^2$ given by the left and right shifts of
functions on $G$ along  $\Gamma$:
\[\rho_{\gamma_1,\gamma_2}(f)(g)=f(\gamma_1^{-1}g\gamma_2)\]
where $\gamma_1,\gamma_2\in\Gamma$, $f\in\C_0(G)$ and $g\in G$. We
construct the crossed product $\C^*$-algebra
$B=\C_0(G)\rtimes_\rho\Gamma^2$. Let $\mathbb{B}=(B,\lambda,\hat\rho)$ be the
standard $\Gamma^2$-product structure on $B$, i.e. $\hat\rho$ is
the dual action of $\Hat\Gamma^2$ on $B$ and $\lambda$ is the
representation of $\Gamma^2$ on $B$, which implements  $\rho$ on
$\C_0(G)\subset\M(B)$:
\[\rho_{\gamma_1,\gamma_2}(f)=\lambda_{\gamma_1,\gamma_2}f\lambda_{\gamma_1,\gamma_2}^*.\]
 Let $\Hat\Psi:\Hat\Gamma\rightarrow\Gamma$ be  the homomorphism
given by $\Hat\Psi(\hat\gamma)=\frac{s}{4}\overline{\hat\gamma}$
for any $\hat\gamma\in\Hat\Gamma$ (note that we used the identifiaction $\Gamma\cong\Hat\Gamma\cong\mb{C}$). Using $\Hat\Psi$ we twist
$\hat\rho$ getting the dual action
$\hat\rho^\Psi:\Hat\Gamma^2\rightarrow\Aut(B)$ (in \cite{Kasp} it
was denoted by $\hat\rho^{\tilde\Psi\ot\Psi}$):
\begin{equation}\label{twistac} \hat\rho^\Psi_{\hat\gamma,\hat\gamma'}(b)=
\lambda_{-\Hat\Psi(\hat\gamma),\Hat\Psi(\hat\gamma')}
\,\hat\rho_{\hg,\hg'}(b)\,
\lambda_{-\Hat\Psi(\hat\gamma),\Hat\Psi(\hat\gamma')}^*.\end{equation}
 As was shown in
\cite{Kasp}, the triple $\mb{B}^\Psi=(B,\lambda,\hat\rho^\Psi)$ is
also a $\Gamma^2$-product. The $\C^*$-algebra $A$ of the
Heisenberg-Lorentz quantum group $\mb{G}$ is defined as the
Landstad algebra of $\mb{B}^\Psi$: $A\subset\M(B)$ is the subalgebra of elements satisfying the Landstad conditions:
 \begin{equation}\label{laninv} \begin{array}{cl} 1.&\hat\rho^\Psi_{\hg,\hg'}(b)=b
\mbox{ for all }\hg,\hg'\in\Hat\Gamma;\\
2. & \mbox{the map }
\Gamma^2\ni(\gamma,\gamma')\mapsto\lambda_{\gamma,\gamma'}b\lambda_{\gamma,\gamma'}^*\in\M(B)
\mbox{ is norm continuous;}\\
3.& xbx'\in B \mbox{ for all }x,x'\in\C^*(\Gamma)\subset\M(B).
\end{array}\end{equation} The three conditions defining $A\subset \M(B)$ will be
refereed to as the Landstad conditions.

 The $\C^*$-algebra $A$
carries the structure of a quantum group. All structure maps can
be described in terms of the $\Gamma^2$-product $\mb{B}$ but in
this paper we shall rather use
 the fact that they are related to a 
multiplicative unitary $W$ of $\mb{G}$, whose construction goes as follows. Let $dg$ be
a right invariant Haar measure on $G$ and let $L^2(G)$ be the Hilbert space of
square integrable functions with respect to $dg$. Let
$L_g,R_g\in\B(L^2(G))$ be the left and right regular
representation of $G$. Restricting them to $\Gamma\subset G$ we
get two representations of $\Gamma$ on $L^2(G)$. The related
representations of $\C^*(\Gamma)$ will be respectively denoted by
$\pi^L,\pi^R\in\Rep(\C^*(\Gamma);L^2(G))$. Obviously
$\Psi\in\M(\C^*(\Gamma)\ot\C^*(\Gamma))$ is unitary, hence
operators  $X,Y\in\B(L^2(G)\ot L^2(G))$ given by
\begin{equation}\label{XYdef}
\begin{array}{rcl}X\hsp{=}(\pi^R\ot\pi^R)(\Psi),\\
Y\hsp{=}(\pi^R\ot\pi^L)(\Psi)
\end{array}
\end{equation}
are unitary too. Finally, the multiplicative unitary
$W\in\B(L^2(G)\ot L^2(G))$ of $\mb{G}$ is of the following
form:
\begin{equation}\label{Wdef}W=YVX\end{equation}
where $V$ is the standard Kac-Takesaki operator of the group $G$. The
$\C^*$-algebra $A$ is isomorphic with the $\C^*$-algebra  of
slices of the first leg of $W$:
\begin{equation}\label{slicw}
A\simeq\left\{(\omega\ot\id)W:\omega\in\B(L^2(G))_*\right\}^{\|\cdot\|-\mbox{closure}}.
\end{equation}
Note that $A$ treated as the algebra of slices of $W$ is naturally
represented on $L^2(G)$.
\begin{subsection}{Affiliated element $\hg$}\label{hg} The idea
of construction of $\hg\in A^\eta$ is based on the observation that the normal operator
 $\gamma\in B^\eta$ is a central element, which by definition means that $z_{\gamma}\in\M(B)$
 is a central element. Together with the invariance of $\gamma$ under $\hat\rho$ this implies the invariance of  $\gamma$ under the twisted dual action $\hat\rho^\Psi$:
 \begin{eqnarray*} \hat\rho^\Psi_{\hat\gamma,\hat\gamma'}(\gamma)&\hspace*{-0.2cm}=
 &\hspace*{-0.2cm}\lambda_{-\Hat\Psi(\hat\gamma),\Hat\Psi(\hat\gamma')}
\,\hat\rho_{\hg,\hg'}(\gamma)\,
\lambda_{-\Hat\Psi(\hat\gamma),\Hat\Psi(\hat\gamma')}^*\\
&\hspace*{-0.2cm}=&\hspace*{-0.2cm}\lambda_{-\Hat\Psi(\hat\gamma),\Hat\Psi(\hat\gamma')}
\,\gamma\,
\lambda_{-\Hat\Psi(\hat\gamma),\Hat\Psi(\hat\gamma')}^*=\gamma\end{eqnarray*}
hence the first Landstad condition defining elements of $A$ is satisfied for $\gamma$.
Not to bother with the other Landstad conditions we give the following construction of
$\hat\gamma\in A^\eta$.
Let $\id:\mb{C}\rightarrow\mb{C}$ be the identity function:
 $\id(z)=z$ for any $z\in\mb{C}$. This function
  generates $\C_0(\mb{C})$ in the sense of
Woronowicz (see Definition 3.1 of \cite{Worgen}). Let
$\pi\in\Mor(\C_0(\mb{C});\C_0(G))$ be the morphism that sends
$\id\in\C_0(\mb{C})^\eta$ to the coordinate function
$\gamma\in\C_0(G)^\eta$. From the invariance of  $\gamma$ under
the action $\rho:\Gamma^2\rightarrow\Aut(\C_0(G))$ it follows that
$\pi$ satisfies the assumptions of Theorem $3.18$ of \cite{Kasp}
for the trivial action of $\Gamma^2$ on $\C_0(\mb{C})$, therefore
it gives rise to the twisted morphism
$\pi^\Psi\in\Mor(\C_0(\mb{C});A)$. We define $\hg\in A^\eta$ as
the image of $\id\in\C_0(\mb{C})^\eta$ under $\pi^\Psi$:
$\hg=\pi^\Psi(\id)$. Obviously, the $z$-transform
$z_{\hg}$ belongs to the center of $M(A)$ and $\hg$ treated as an
operator acting  on $L^2(G)$ (c.f. \eqref{slicw}) coincides with
the operator of multiplication  by the coordinate $\gamma$.
\end{subsection}
\begin{subsection}{Affiliated elements $\ha$ and
$\hd$}\label{ha} As has already been mentioned, the couples $(\ha,-s\hg^*\hg)$ and $(\hd,s\hg^*\hg)$
satisfy the Heisenberg Lie algebra relation (see Appendix \ref{bcg}).
This observation motivates the idea of defining $\ha$ and $\hd$ as complex infinitesimal generators of appropriately defined representations $U^{\ha}$ and $U^{\hd}$ of the Heisenberg group $\mathbb{H}$ on $A$.
In what follows we shall show the fruitfulness of this approach but first we need to introduce some notation.

Let $T$ be a normal
element affiliated with a $\C^*$-algebra $C$. By $\lambda(z;T)$ we
shall denote the unitary given by
\begin{equation}\label{replam}\lambda(z;T)=\exp(i\Im(zT))\in \M(C).\end{equation} Note that the map
\[\mb{C}\ni z\mapsto \lambda(z;T)\in\M(C)\] is a representation of the
additive group $\mb{C}$. This representation will be denoted by
$\lambda(\,\cdot\,;T)$.
 Let $\mb{B}=(B,\lambda,\hat\rho)$ be the $\Gamma^2$-product introduced
 at the beginning of Section \ref{calglev}. Let $T_l,T_r\in B^\eta$ be infinitesimal generators of representation
$\lambda:\mb{C}^2\rightarrow \M(B)$:
\begin{equation}\label{infgenl}\lambda_{z_1,z_2}=\lambda(z_1,T_l)\lambda(z_2,T_r). \end{equation}
The coordinate function $\gamma\in \C_0(G)^\eta\subset B^\eta$ - being central -
commutes with $\lambda_{z_1,z_2}$:
\[\lambda_{z_1,z_2}\gamma\lambda_{z_1,z_2}^*=\gamma.\] This
implies that $T_l,T_r$ and $\gamma$ strongly commute in the sense
of Definition \ref{c1}, hence using Theorem \ref{thmc2} we may construct a
pair of normal elements: $\gamma T^*_l,\gamma T^*_r\in B^\eta$.
We define the aforementioned representations $U^{\ha}$ and $U^{\hd}$
of $\mb{H}$ in $B$ by the following formulas:
 \begin{equation}\label{repdef}\begin{array}{rcl}U^{\ha}_{z,t}\hsp{=}\lambda
 \left(z;\alpha\right)\lambda\left(z;-\frac{s}{4}\gamma T^*_l\right)
 \lambda\left(t;-\frac{s}{4}\gamma^*\gamma\right),\\
 U^{\hd}_{z,t}\hsp{=}\lambda
 \left(z;\delta\right)\lambda\left(z;-\frac{s}{4}\gamma T^*_r\right)
 \lambda\left(t;\frac{s}{4}\gamma^*\gamma\right)\end{array}
 \end{equation} were we used the embedding $\C_0(G)^\eta\subset B^\eta$
 to treat  $\alpha,\delta$ and $\gamma^*\gamma$ as normal elements affiliated with $B$.
The analysis of $U^{\ha}$ and $U^{\hd}$ will be performed in the proof of the next theorem
but before we formulate it, we have to introduce some auxiliary objects.
Let $\partial_l$ and
$\partial_r$ denote the vector fields  whose actions on a function $f\in\C^\infty(G)$ are given by:
  \begin{equation}\label{diffop1}\begin{array}{rcl}\partial_lf(g)\hsp{=}2\frac{\partial}{\partial
  z}L_z(f)(g)|_{z=0},\vspace{0,1cm}\\
\partial_rf(g)\hsp{=}2\frac{\partial}{\partial z}R_z(f)(g)|_{z=0}\end{array}
 \end{equation}
($L_z$ and $R_z$  denote the
operators of the left and right shift by an element $z\in\Gamma$).
Using $\partial_l$ and $\partial_r$ we define differential
operators $\Op(\alpha)$ and $\Op(\delta)$ acting on
$\C^\infty_c(G)$:
\begin{equation}\label{diffop}\begin{array}{rcl}\Op(\alpha)\hsp{=}\alpha-\frac{s}{4}\gamma\partial_l^*,\\
 \Op(\delta)\hsp{=}\delta-\frac{s}{4}\gamma\partial_r^*\end{array}.
 \end{equation}
  The quantization map $\mc{Q}$ used
in the next theorem is described in Appendix \ref{Aqm}.
\begin{thm}\label{repad} Let $U^{\ha}_{z,t},U^{\hd}_{z,t}\in\M(B)$ be the unitary
elements given by \eqref{repdef}. Let $\Op(\alpha)$ and
$\Op(\delta)$ be the differential operators given by
\eqref{diffop} and let $A$ be
 the Landstad algebra of $\mb{B}^\Psi$ (see \eqref{laninv}).
Then:
\begin{itemize}
\item[1.] $U^{\ha}_{z,t},U^{\hd}_{z,t}$ are elements of $\M(A)\subset\M(B)$;
\item[2.] the maps
\begin{eqnarray*} \mb{H}\ni(z,t)\mapsto U^{\ha}_{z,t}\in\M(A)\\
                  \mb{H}\ni(z,t)\mapsto U^{\hd}_{z,t}\in\M(A)
                  \end{eqnarray*}
are strictly continuous, commuting representations of the
Heisenberg group;
\item[3.] let $\ha,\hd\in A^\eta$ be the complex generators of $U^{\ha}$ and $U^{\hd}$ (see Appendix \ref{bcg}).
The set $\{\mc{Q}(f):f\in\C_c^\infty(G) \}\subset A$ is a common
core of $\ha,\hd\in A^\eta$ and we have
\begin{equation}\label{acth1}\begin{array}{c}\ha\mc{Q}(f)=\mc{Q}(\Op(\alpha)f)\\
                 \hd\mc{Q}(f)=\mc{Q}(\Op(\delta)f)
\end{array}\end{equation} for any $f\in\C_c^\infty(G)$.
\end{itemize}
\end{thm}
\begin{proof} Let us begin by proving that
$U^{\ha}_{z,t},U^{\hd}_{z,t}\in\M(A)$. Let $\hat\rho^\Psi$  be the
twisted dual action given by \eqref{twistac}. It is easy to check that
\[\left\{\begin{array}{ll}1.& \hat\rho^\Psi_{z_1,z_2}\left(\lambda
 \left(z;\alpha\right)\right)=\lambda\left(z,\frac{s}{4}\bar{z}_1\gamma\right)\lambda
 \left(z;\alpha\right),\vspace{0.1cm}\\
2.&\hat\rho^\Psi_{z_1,z_2}\left(\lambda\left(z;-\frac{s}{4}\gamma
T^*_l\right)\right)=\lambda(z;-\frac{s}{4}\bar{z}_1\gamma)\lambda\left(z;-\frac{s}{4}\gamma
T^*_l\right),\vspace{0.1cm}\\3.&\hat\rho^\Psi_{z_1,z_2}\left(\lambda
 \left(z;\delta\right)\right)=\lambda\left(z;\frac{s}{4}\bar{z}_2\gamma\right)\lambda
 \left(z;\delta\right),\vspace{0.1cm}\\4.&\hat\rho^\Psi_{z_1,z_2}\left(\lambda\left(z;-\frac{s}{4}\gamma
T^*_r\right)\right)=\lambda(z;-\frac{s}{4}\bar{z}_2\gamma)\lambda\left(z;-\frac{s}{4}\gamma
T^*_r\right).
 \end{array}\right.\]
 Using these equalities and \eqref{repdef} we see that
  $U^{\ha}_{z,t}$ and
 $U^{\hd}_{z,t}$ satisfy the first Landstad condition (see \eqref{laninv}).
  Let us move on to the second Landstad condition.
 One can check that:
\begin{equation}\label{gan}\begin{array}{l}\lambda_{z_1,z_2}U^{\ha}_{z,t}\lambda_{z_1,z_2}^*=\lambda(z,-z_1\gamma)U^{\ha}_{z,t},\vspace{0,1cm}\\
\lambda_{z_1,z_2}U^{\hd}_{z,t}\lambda_{z_1,z_2}^*=\lambda(z,z_2\gamma)U^{\ha}_{z,t}.\end{array}\end{equation}
In Section \ref{hg} we constructed the affiliated element $\hg\in
A^\eta$. Its image in $B^\eta$ coincides with the coordinate
function $\gamma\in\C_0(G)^\eta\subset B^\eta$, hence  using \eqref{gan} we see that the
following two maps are norm continuous:
\[\begin{array}{l}\mb{C}^2\ni(z_1,z_2)\mapsto\lambda_{z_1,z_2}U^{\ha}_{z,t}\lambda_{z_1,z_2}^*a\in\M(B)\\
\mb{C}^2\ni(z_1,z_2)\mapsto\lambda_{z_1,z_2}U^{\hd}_{z,t}\lambda_{z_1,z_2}^*a\in\M(B).
                 \end{array}\]
for any $a\in A$. This norm-continuity together with
$\hat\rho^\Psi$-invariance of
$U^{\ha}_{z,t},U^{\hd}_{z,t}\in\M(B)$ implies that
$U^{\ha}_{z,t},U^{\hd}_{z,t}$ are indeed elements of $\M(A)$. The
proof that $U^{\ha}$ and $U^{\hd}$ are commuting representations
of the Heisenberg group $\mb{H}$ is left to the reader. 

Let us now prove the strict
continuity of these representations. For this purpose we shall
treat $U^{\ha}_{z,t}$ and  $U^{\hd}_{z,t}$ as operators acting on
$L^2(G)$ (see \eqref{slicw}). It can be checked that the action of $U^{\ha}$ and
$U^{\hd}$ on $L^2(G)$ expressed in coordinates
$\alpha,\delta,\gamma$ is given by:
\begin{equation}\label{l2u}
\begin{array}{rcl}U^{\ha}_{z,t}f(\alpha,\gamma,\delta)\hsp{=}
\exp\left(-i\frac{st}{4}\bar{\gamma}\gamma\right)\exp(i\Im(z\alpha))
f(\alpha-\frac{s}{4}\bar{z}\bar{\gamma}\gamma,\gamma,\delta)\\
U^{\hd}_{z,t}f(\alpha,\gamma,\delta)\hsp{=}
\exp\left(i\frac{st}{4}\bar{\gamma}\gamma\right)\exp(i\Im(z\delta))
f(\alpha,\gamma,\delta+\frac{s}{4}\bar{z}\bar{\gamma}\gamma)
\end{array}
\end{equation} for any $f\in L^2(G)$. Using Theorem 4.14 of \cite{Kasp} we obtain:
\begin{equation}\label{uq}U^{\ha}_{z,t}\mc{Q}(f)=\mc{Q}(U^{\ha}_{z,t}f)\end{equation} for any
$f\in\C_c^{\infty}(G)$. This together with Theorem \ref{A1} leads
to the following estimation:
\[\|U^{\ha}_{z,t}\mc{Q}(f)-\mc{Q}(f)\|\leq c\max_{k,k',l,l'\leq 5}
\sup_{g\in
G}\left|\partial_l^{k*}\partial_l^{k'}\partial_r^{*m}\partial_r^{m'}(U^{\ha}_{z,t}f-f)\right|.\]
Using \eqref{l2u} we may see that the right hand side of the above inequality is convergent to zero when $(z,t)\rightarrow (0,0)$, which shows
that:
\[\lim_{(z,t)\rightarrow
(0,0)}\|U^{\ha}_{z,t}\mc{Q}(f)-\mc{Q}(f)\|=0.\] The density of the
set $\{\mc{Q}(f):f\in\C_c^\infty(G)\}$ in $A$ ensures that
$U^{\ha}$ is continuous  in the sense of the  strict topology of
$\M(A)$. We can prove the strict continuity of the
representation
 $U^{\hd}$ in a similar way.

Let us move on to the proof of the third point of our theorem.
Using equation \eqref{l2u} one can check that
\[2\frac{\partial}{\partial
z}U^{\ha}_{z,0}f\Big|_{z=0}=\Op(\alpha)f\] for any $f\in
\C_c^\infty(G)$. This together with \eqref{uq} and \eqref{remdif1}
proves the first equation of
 \eqref{acth1}.
 The second formula of \eqref{acth1} is proven in a similar way.
To prove that $\mc{Q}(\C_c^\infty(G))$ is a core of either $\ha$
or
 $\hd$ it is enough to check that the sets
 \begin{equation}\label{opaden}\begin{array}{c}\{(1+\ha^*\ha)\mc{Q}(f):f\in\C_c^\infty(G)\},\\
 \{(1+\hd^*\hd)\mc{Q}(f):f\in\C_c^\infty(G)\}
 \end{array}\end{equation} are dense in $A$ (see Lemma \ref{lemcore}).
In what follows we shall sketch the proof of the density for the
first of these sets. Let us first note that the set
\[(1+\ha^*\ha)^{-1}\mc{Q}(\C_c^\infty(G))=\{(1+\ha^*\ha)^{-1}\mc{Q}(f):f\in\C_c^\infty(G)\}\subset
A\] is dense in $A$. This follows from the density of
$\mc{Q}(\C_c^\infty(G))$ in $A$ and the fact that $\ha\in A^\eta$.
Let $f$ be an arbitrary element of $\C_c^\infty(G)$ and $g\in
L^2(G)$ the function given by:
\[g=(1+\ha^*\ha)^{-1}f.\]
Using \eqref{asa} we see that
\begin{equation}\label{dena}g=\int_{\mb{R}_+}dt\,\exp(-t)\int_{\mb{C}}
d^2z\,h_t\left(z,-s\frac{1}{2}\hg^*\hg\right)U^{\ha}_{z,0}f.\end{equation}
One can check that $g$ is quantizable in the sense of Theorem
\ref{A1}, $\mc{Q}(g)\in\D(\ha^*\ha)$ and
$\mc{Q}(f)=(1+\ha^*\ha)\mc{Q}(g)$. Using formula \eqref{dena} we
can prove the existence of a sequence  $f_n\in\C_c^\infty(G)$ such
that
\begin{equation}\label{parcon}\begin{array}{rcl}\displaystyle\lim_{n\rightarrow
\infty}\partial_l^k\partial_l^{*k'}\partial_r^m\partial_r^{*m'}f_n
\hsp{=}\partial_l^k\partial_l^{*k'}\partial_r^m\partial_r^{*m'}g\\
\displaystyle\lim_{n\rightarrow
\infty}\partial_l^k\partial_l^{*k'}\partial_r^m\partial_r^{*m'}(1+\Op(\alpha)^*\Op(\alpha))f_n
\hsp{=}\partial_l^k\partial_l^{*k'}\partial_r^m\partial_r^{*m'}(1+\Op(\alpha)^*\Op(\alpha))g\end{array}\end{equation}
for any $k,k',m,m'\leq 5$. By equations \eqref{acth1},
\eqref{parcon}, Theorem \eqref{A1} and the closedness of $\ha$ we
get
\[\mc{Q}(f)=\lim_{n\rightarrow\infty}(1+\ha^*\ha)\mc{Q}(f_n).\]
Using the fact that  $f$ is an arbitrary smooth function of
compact support and $\mc{Q}(\C_c^\infty(G))$ is a dense subset of
$A$ we get:
\[\overline{(1+\ha^*\ha)\mc{Q}(\C_c^\infty(G))}^{||\cdot||}=A.\]
This ends the proof of \eqref{opaden} for
$\ha$.
\end{proof}
\end{subsection}
\begin{subsection}{Quantum Borel subgroup $\mb{G}_0$}\label{qbs}
This section is a preparation for the construction of the affiliated element $\hb\in A^\eta$.
For this purpose we have to split $A$ into appropriately defined
quantum subspaces. This splitting corresponds to the splitting of the classical group $G$ into its
Borel subgroup $G_0\subset G$ and the set theoretic complement of $G_0$, where by  $G_0\subset G$ we understand:
\[G_0=\left\{\left(
           \begin{array}{cc}
             \alpha_0 & \beta_0 \\
             0 & \alpha^{-1}_0\\
           \end{array}
         \right):\alpha_0\in\mb{C}_*, \beta_0\in\mb{C}\right\}.\]
Let us be more precise.
Let $\pi_0\in\Mor(\C_0(G);\C_0(G_0))$ be the restriction morphism:
\[\pi_0(f)(g_0)=f(g_0)\]
for any $f\in\C_0(G)$ and $g_0\in G_0$.
  Applying the Rieffel
deformation to $(\C_0(G_0),\Delta)$, based on the subgroup
$\Gamma\subset G$ we construct a quantum group $\mb{G}_0=(A_0,\Delta)$. Let
$\mb{B}_0$  be the respective $\Gamma^2$-product. In this section
$T_l$ and $T_r$ denote the infinitesimal generators of the
representation $\lambda:\mb{C}^2\rightarrow \M(B_0)$ (see
\eqref{infgenl}) and $\partial_l,\partial_r$ denote the
 vector fields on $G_0$ defined like in \eqref{diffop1}.
 By Theorem $3.18$ of
\cite{Kasp} the restriction morphism $\pi_0
\in\Mor(\C_0(G);\C_0(G_0))$ induces the twisted morphism of
$\C^*$-algebras $\pi_0^\Psi:A\rightarrow A_0$ and the surjectivity
of $\pi_0$ implies the surjectivity of $\pi_0^\Psi$. Let
$A_{\hg}\subset A$ be the two-sided ideal generated by
 $z_{\hg}$. Invoking the centrality of  $z_{\hg}$ in $\M(A)$ we
 have  $A_{\hg}= \overline{z_{\hg}A}^{\|\cdot\|}$.
 It is easy to see that $\pi_0^{\Psi}(z_{\hg})=0$, which implies that
 $A_{\hg}\subset\ker\pi_0^{\Psi}$.  It can also be proven that $\ker\pi_0^{\Psi}\subset A_{\hg}$, hence
we have the exact sequence of $\C^*$-algebras:
\begin{equation}\label{exseq}0\rightarrow A_{\hg}\rightarrow
A\xrightarrow{\pi_0^\Psi}A_0\rightarrow 0.\end{equation}

In what follows we shall construct an affiliated element $\hb_0\in
A_0^\eta$, which is farther used in the construction of $\hb\in
A^\eta$. Let us first mention that following the construction of
$\hg\in A^\eta$ of Section \ref{hg}, we may introduce an
affiliated element $\ha_0\in A^\eta_0$. It is
 normal and invertible and its action on $L^2(G_0)$ is given by the
 multipliciation operator by the coordinate $\alpha_0$.
Remebering that $\C_0(G_0)^\eta\subset B_0^\eta$ we shall consider
$\alpha_0$ and $\beta_0$ as affiliated with $B_0$. The elements
$\alpha_0, T_l, T_r\in B_0^\eta$ strongly commute, hence using Theorem
\ref{thmc2} we construct $\alpha_0T_r^*,\,\alpha^{-1}_0T^*_l\in
B_0^\eta$. For any $(z,t)\in\mb{H}$
 we define the unitary element:
 \begin{equation}\label{repdef1}U^{\hb_0}_{z,t}=\lambda
 \left(z;\beta_0\right)\lambda\left(z;-\frac{s}{4}\alpha_0^{-1} T^*_l\right)
 \lambda\left(z;-\frac{s}{4}\alpha_0 T^*_r\right)
 \lambda\left(t;-\frac{s}{4}(|\alpha_0|^{-2}-|\alpha_0|^{2})\right)\in\M(B_0).
 \end{equation} Let us also define the differential operator:
 \begin{equation}\Op(\beta_0)=\beta_0-\frac{s}{4}\alpha_0^{-1}\partial_l^*-\frac{s}{4}\alpha_0\partial_r^*.\end{equation}
The proof of the next theorem is similar to the proof of Theorem
\ref{repad}.  The quantization map  related to the quantum group
$\mb{G}_0$ is denoted by $\mc{Q}_0$.
\begin{thm} Let $U^{\hb_0}_{z,t}\in\M(B_0)$ be the unitary
element given by  formula \eqref{repdef1}. Then
\begin{itemize}
\item[1.] $U^{\hb_0}_{z,t}$ is an element of $\M(A_0)\subset\M(B_0)$ for any $(z,t)\in\mb{H}$;
\item[2.] the map
\[\mb{H}\ni(z,t)\mapsto U^{\hb_0}_{z,t}\in\M(A_0)\]
is a strongly continuous representation of the Heisenberg group;
\item[3.] the set $\{\mc{Q}_0(f):f\in\C_c^\infty(G_0)
\}\subset A_0$ is a core of the generator $\hb_0\in A_0^\eta$ of the
representation $U^{\hb_0}$  and we have
\begin{equation}\label{acth}\hb_0\mc{Q}_0(f)=\mc{Q}_0(\Op(\beta_0)f)
\end{equation} for any $f\in\C_c^\infty(G_0)$. Moreover the set
\begin{equation}\label{densbo}
\{\mc{Q}_0((1+\Op(\beta_0)^*\Op(\beta_0))f):f\in\C_c^\infty(G_0)\}\end{equation}
is dense in $A_0$.
\end{itemize}
\end{thm}
\end{subsection}
\begin{subsection}{Affiliated element $\hb$}
After constructing the affiliated element $\hb_0\in A_0$ we shall now move on to the construction of $\hb$.
In order to do that we first have to introduce $\hb_{\hg}\in A_{\hg}^\eta$ which may be treated as a restriction
of $\hb$ to $A_{\hg}$.
The embedding of $A_{\hg}$ into $A$ leads to a morphism
$\pi_{\hg}\in\Mor(A,A_{\hg})$ which is defined by the
formula $\pi_{\hg}(a)a_{\hg}=aa_{\hg}$ where $a\in A$ and
$a_{\hg}\in A_{\hg}$. This morphism is injective, which enables us
to treat $\ha,\hg,\hd\in A^\eta$ as elements affiliated with
$A_{\hg}$. The injectivity of $\pi_{\hg}$ follows from the
 implication
$\left(az_{\hg}=0\right)\Rightarrow \left(a=0\right)$, which is
true for any $a\in A$. Note that $\hg$ treated as an element of
$A_{\hg}$ is invertible, i.e. there exists a unique element
$\hg^{-1}\in A_{\hg}^\eta$ strongly commuting with $\hg$ and such
that $\hg\hg^{-1}=\hg^{-1}\hg=1$. Moreover, elements
$\ha,\hd,\hg^{-1}\in A_{\hg}^\eta$ mutually strongly commute so
using Theorem \ref{thmc2} we may define $\hb_{\hg}$ by the formula
\begin{equation}\label{fordefb}\hb_{\hg}=\hg^{-1}(\ha\hd-1)\in A_{\hg}^\eta.\end{equation}

In order to give a more direct description of $\hb_{\hg}$, let us
introduce the differential operator:
\begin{equation}\label{opb}
\Op(\beta)=\beta-\frac{s}{4}\delta\partial^*_l-\frac{s}{4}\alpha\partial^*_r+
\frac{s^2}{16}\gamma\partial^*_l\partial^*_r.\end{equation} It is
easy to check that the determinant relation is satisfied
\[\Op(\alpha)\Op(\delta)-\Op(\gamma)\Op(\beta)=1\] where
$\Op(\gamma)$ denotes the operator of multiplication by $\gamma$.
The following lemma describes $\hb_{\hg}$ in terms of
$\Op(\beta)$.
\begin{lem}\label{hbg} Let $\hb_{\hg}\in A_{\hg}^\eta$ be the affiliated
element defined above. The set
\[\{\mc{Q}(f)z_{\hg}:f\in\C_c^\infty(G)\}\] is a core of
$\hb_{\hg}$ and  for any $f\in\C_c^\infty(G)$ we have
\begin{equation}\label{actbg}\hb_{\hg}\mc{Q}(f)z_{\hg}=\mc{Q}(\Op(\beta)f)z_{\hg}.
\end{equation} Moreover, the set
\begin{equation}\label{dbg}\{(1+\hb_{\hg}^*\hb_{\hg})\mc{Q}(f)z_{\hg}:f\in\C_c^\infty(G)\}\end{equation}
is dense in $A_{\hg}$.
\end{lem}
\begin{proof} Formula \eqref{actbg} follows  from equation \eqref{fordefb}
 and point 3 of Theorem \ref{repad}.
In order to prove \eqref{dbg} we
introduce the affiliated element $T=\ha\hd-1\in A_{\hg}^\eta$.
Obviously, we have $\hb_{\hg}=\hg^{-1}T$ so we will base the analysis of $\hb_{\hg}$  on the analysis of $T$. Let
us check that $T$ and $\ha^*\ha+\hd^*\hd$ strongly commute:
\begin{equation}
\begin{array}{rcl}
\exp(it(\ha^*\ha+\hd^*\hd))T\exp(it(\ha^*\ha+\hd^*\hd))&&\\\hsp{\hspace*{-4cm}=}\hspace*{-2cm}
\exp(it\ha^*\ha)\,\ha\exp(-it\ha^*\ha)\exp(it\hd^*\hd)\,\hd\exp(-it\hd^*\hd)-1\\
\hsp{\hspace*{-4cm}=}\hspace*{-2cm}\exp(its\hg^*\hg)\,\ha\exp(-its\hg^*\hg)\,\hd-1=T.
\end{array}
\end{equation}
Using Theorem \ref{c1} we define
\begin{equation}\label{tpr}T'=(1+T^*T)\exp(-\ha^*\ha-\hd^*\hd)\in A_{\hg}^\eta.\end{equation}
The equality\,\,\,
$\displaystyle\exp(-\ha^*\ha-\hd^*\hd)=\exp(-\ha\ha^*-\hd\hd^*)$  \,\,\,implies
 that
\begin{equation}\label{sumt}\begin{array}{rcl}T'\hsp{=}2\exp(-\ha^*\ha-\hd^*\hd)+\ha^*\ha\exp(-\ha^*\ha)\hd^*\hd\exp(-\hd^*\hd)
\\\hsp{-}\ha\exp(-\ha^*\ha)\hd\exp(-\hd^*\hd)-\ha^*\exp(-\ha\ha^*)\hd^*\exp(-\hd^*\hd).\end{array}\end{equation}
All factors of the above sum belong to $\M(A_{\hg})$, hence the resulting operator 
 $T'$ also belongs to $\M(A_{\hg})$. Note that
\begin{equation}\label{sumt1}T'\D(T^*T)=\exp(-\ha^*\ha-\hd^*\hd)(1+T^*T)\D(T^*T)=
\exp(-\ha^*\ha-\hd^*\hd)A_{\hg}.\end{equation} The right-hand
side of \eqref{sumt1} is dense in $A_{\hg}$, hence using the boundedness
of $T'$ and the density of $\D(T^*T)$ in $A_{\hg}$ we conclude
that the set $T'\mc{Q}(\C_c^\infty(G))z_{\hg}$ is also dense:
\begin{equation}\label{dent1}\overline{T'\mc{Q}(\C_c^\infty(G))z_{\hg}}^{\|\cdot\|}=A_{\hg}.\end{equation} 

We shall now prove that the density \eqref{dent1} implies the density \eqref{dbg}. 
Let
\begin{equation}\label{aintr} a=\exp(-\ha^*\ha-\hd^*\hd)\mc{Q}(f)z_{\hg}\end{equation} for some
$f\in\C_c^\infty(G)$. Using  formula \eqref{asa} one can check that there exists a sequence
$f_n\in\C_c^\infty(G)$ such that
\begin{equation}\label{parcon1} \begin{array}{c}\displaystyle\lim_{n\rightarrow
\infty}\partial_l^k\partial_l^{*k'}\partial_r^m\partial_r^{*m'}f_n
z_{\gamma}
=\partial_l^k\partial_l^{*k'}\partial_r^m\partial_r^{*m'}\exp(-\ha^*\ha-\hd^*\hd)fz_{\gamma}\\
\displaystyle\lim_{n\rightarrow
\infty}\partial_l^k\partial_l^{*k'}\partial_r^m\partial_r^{*m'}(1+T^*T)f_n
z_{\gamma}
=\partial_l^k\partial_l^{*k'}\partial_r^m\partial_r^{*m'}(1+T^*T)\exp(-\ha^*\ha-\hd^*\hd)fz_{\gamma}\end{array}\end{equation}
for any $k,k',m,m'\leq 5$, where in the above formulas   we treat
$T=\ha\hd-1$ as an operator acting on
$L^2(G)$. It may be shown  that the convergence in \eqref{parcon1} is in the
uniform topology on $\C_0(G)$. Using  Theorem
\ref{A1} and the closedness of $1+T^*T$ we see that
\begin{equation}\label{dent2}(1+T^*T)a=\lim_{n\rightarrow\infty}(1+T^*T)\mc{Q}(f_n)z_{\hg}.\end{equation}
Combining  \eqref{dent1}, \eqref{aintr} and \eqref{dent2} we get
\[\displaystyle\overline{(1+T^{*}T)\mc{Q}(\C_c^\infty(G))z_{\hg}}^{\|\cdot\|}=A_{\hg}.\]
Using the above equality, the fact that  $\hb_{\hg}=\hg^{-1}T$ and Lemma \ref{lemc5} we
get
\begin{equation}\label{fina}\overline{(1+\hb_{\hg}^*\hb_{\hg})(1+|\hg|^{-2})^{-1}\mc{Q}(\C_c^\infty(G))z_{\hg}}^{\|\cdot\|}=A_{\hg}.\end{equation}
The inclusion
\[(1+|\hg|^{-2})^{-1}\mc{Q}(\C_c^\infty(G))z_{\hg}\subset
\mc{Q}(\C_c^\infty(G))z_{\hg}\] and equation \eqref{fina} show that
\[\overline{(1+\hb_{\hg}^*\hb_{\hg})\mc{Q}(\C_c^\infty(G))z_{\hg}}^{\|\cdot\|}=A_{\hg}.\]
This proves \eqref{dbg}. Now from  Lemma \ref{lemcore} it follows that $\mc{Q}(\C_c^\infty(G))z_{\hg}$ is a core of $\hb_{\hg}$ which ends the proof of our
lemma.
\end{proof}
Using $\hb_{\hg}\in A_{\hg}^\eta$ defined above and
 $\hb_0\in A_0^\eta$ defined in the previous section,
we construct the affiliated element $\hb\in A^\eta$.
Heuristically speaking, it is a gluing of $\hb_{\hg}$ and $\hb_0$.
\begin{thm} Let $\Op(\beta)$ be the differential operator  \eqref{opb}.
There exists an affiliated element $\hb\in A^\eta$ such that the
set $\{\mc{Q}(f):f\in\C_c^\infty(G)\}$ is a core of $\hb$ and
\begin{equation}\hb\mc{Q}(f)=\mc{Q}(\Op(\beta)f)\end{equation} for
any $f\in\C_c^\infty(G)$.
\end{thm}
\begin{proof}
Let $\Graph{\hb_{\hg}}$ be the graph of the
affiliated element $\hb_{\hg}$. It is easy to check that the set:
\[\left\{\begin{pmatrix}
  b \\
  b'
\end{pmatrix}:\begin{pmatrix}
  b\,a_{\hg} \\
  b'a_{\hg}
\end{pmatrix}\in\Graph(\hb_{\hg}),  \mbox{ for any }a_{\hg}\in A_{\hg}\right\}\subset A\oplus A\]
is a graph of a closed operator acting on $A$. This operator will
be denoted by $\hb$. Let us list some properties of $\Graph{\hb}$:
 \begin{itemize}
\item[1.]
$\Graph{\hb}\subset A\oplus A$ is a submodule of a Hilbert
$A$-module $A\oplus A$\vspace{0,25cm};
\item[2.] For any $f\in\C_c^\infty(G)$ we have
$\begin{pmatrix}
  \mc{Q}(f) \\
  \mc{Q}(\Op(\beta)f)
\end{pmatrix}\in\Graph{\hb}$;
\item[3.] Let $(\Graph{\hb})^\perp$ be the
submodule perpendicular to $\Graph{\hb}$:
\[(\Graph{\hb})^\perp=\left\{\begin{pmatrix}
  c \\
  c'
\end{pmatrix}:c^*a+c'^*a'=0 \mbox{ for any } \begin{pmatrix}
  a \\
  a'
\end{pmatrix}\in\Graph\hb\right\}.\]
For any $f\in\C_c^\infty(G)$ we have $\begin{pmatrix}
  -\mc{Q}(\Op(\beta)^*f) \\
  \mc{Q}(f)
\end{pmatrix}\in(\Graph{\hb})^\perp$;
\item[4.]
$\overline{\{\mc{Q}((1+\Op(\beta)^*\Op(\beta))f):f\in\C_c^\infty(G)\}}^{\|\cdot\|}=A$.
\end{itemize}
Properties 1, 2 and 3 are consequnces of the definition of $\hb$
and Lemma \ref{hbg}. In what follows we shall prove the last of
the above
properties:\[\overline{\{\mc{Q}((1+\Op(\beta)^*\Op(\beta))f):f\in\C_c^\infty(G)\}}^{\|\cdot\|}=A.\]
Let $a\in A$ and $\pi^\Psi_0\in\Mor(A,A_0)$ be the morphism entering the exact sequence \eqref{exseq}. Using \eqref{densbo} we can see that there exists a
sequence $\tilde{f}_n\in\C_c^\infty(G_0)$ such that
\begin{equation}\label{qi}\pi^\Psi_0(a)=\lim_{n\rightarrow
\infty}\mc{Q}_0((1+\Op(\beta_0)^*\Op(\beta_0))\tilde{f}_n).\end{equation}
Let $f_n\in\C_c^\infty(G)$ be an extension of $\tilde{f}_n$ to the
whole group $G$ and let $\pi_0\in\Mor(\C_0(G),\C_0(G_0))$ be the morphism
introduced in Section \ref{qbs}. It is not difficult to check that
\begin{itemize}
\item $\pi_0^\Psi(\mc{Q}(f))=\mc{Q}_0(\pi_0(f))$,
\item
$\pi_0(\Op(\beta)f)=\Op(\beta_0)\pi_0(f)$ \end{itemize}
 Using these equalities and \eqref{qi}  we see
that \begin{equation}\label{conf}\lim_{n\rightarrow
\infty}\pi^\Psi_0(a-\mc{Q}((1+\Op(\beta)^*\Op(\beta))f_n))=0.\end{equation}
The exactness of  sequence \eqref{exseq} ensures
 that for any $\ep>0$ there exists $n\in\mb{N}$ and $a_{\hg}\in
A_{\hg}$ such that
\begin{equation}\label{ine1}\|a-\mc{Q}((1+\Op(\beta)^*\Op(\beta))f_n)-a_{\hg}\|\leq\ep.\end{equation}
Equality \eqref{dbg} implies that there exists a function
$f\in\C_c^\infty(G)$ such that
\begin{equation}\label{ine2}\|a_{\hg}-\mc{Q}((1+\Op(\beta)^*\Op(\beta))f)z_{\hg}\|\leq\ep.\end{equation}
Combining  \eqref{ine1} and \eqref{ine2}  we get
\[\|a-\mc{Q}((1+\Op(\beta)^*\Op(\beta))(f_n+fz_{\gamma}))\|\leq2\ep.\]
This ends the proof of point 4 of our list.

Using the  properties of  $\Graph{\hb}$ one can  check that it
satisfies all the assumptions of Proposition 2.2 of \cite{Worun}.
This proposition guaranties that $\hb\in A^\eta$. It is easy to
check that so constructed $\hb$ satisfies all the requirements of
our theorem.

\end{proof}
\end{subsection}
\begin{subsection}{Representation theory of $\C^*$-algebra
$A$} The results of Appendix \ref{ApD} applied to the
$\C^*$-algebra $A$ of the Heisenberg-Lorentz quantum group
$\mb{G}$ show that the representation theory of $A$ can be
equivalently described by the corepresentation theory of the dual
quantum group $\widehat{\mb{G}}$. As was shown in \cite{Kasp}, the $\C^*$-algebra of $\widehat{\mb{G}}$ is the reduced group $\C^*$-algebra $\C_r^*(G)$. The comultiplication $\Delta_{\widehat{\mb{G}}}$ is the $2$-cocycle twist of the standard comultiplication $\Hat\Delta$ on $\C_r^*(G)$ \begin{equation}\label{twcomul}\Delta_{\widehat{\mb{G}}}(a)= X\Hat\Delta(a)X^*\end{equation} for any $a\in\C_r^*(G)$. The unitary  $X\in\M(\C_r^*(G)\otimes \C_r^*(G))$ is the image of $\Psi\in\M(\C^*(\Gamma)\otimes\C^*(\Gamma))$ (see equation \eqref{dualbi}) under a morphism which sends the generator $u_\gamma\in\M(\C^*(\Gamma))$ to the right shift $R_\gamma\in\M(\C_r^*(G))$. 

Let $\pi_U\in\Rep(A;\mc{H})$ be a representation of $A$ on a Hilbert space $\mc{H}$:
$\pi_U\in\Rep(A;\mc{H})$. The corresponding corepresentation
$U_{\pi}\in\M(\mathcal{K}(H)\otimes \hat A)$ is given by
\begin{equation}\label{corupi}U_\pi=(\pi_U\ot\id)\widehat{W}\end{equation} where $\widehat{W}\in\M(A\ot\Hat{A})$
is the multiplicative unitary of $\widehat{\mb{G}}$. On the other hand, giving a motivation for Definition \ref{repquar} we claimed that representations of the Heisenberg-Lorentz commutation relations correspond to representations of $A$ on Hilbert spaces. To prove this fact we will show that for any representation $(\ta,\tb,\tg,\td)$ on a Hilbert space $\mc{H}$ we can construct a corepresentation $U$ of $\widehat{\mb G}$ on $\mc{H}$, which in turn corresponds via \eqref{corupi} to  a representation $\pi\in\Rep(A;\mc{H})$. This construction of $\pi$ is performed in the proof of the next theorem, were we also give a more direct characterization of $\pi$ in terms of $\ha,\hb,\hg,\hd\in A^\eta$.
\begin{thm}\label{repth} Let $(\ta,\tb,\tg,\td)$ be a representation
of the Heisenberg-Lorentz commutation relations on a Hilbert space
$\mc{H}$ (c.f. Definition \ref{repquar}). There exists a unique
representation $\pi\in\Rep(A;\mc{H})$ such that:
\begin{equation}\label{indrep}\pi(\ha)=\ta,\quad \pi(\hb)=\tb,\quad  \pi(\hg)=\tg,\quad
\pi(\hd)=\td.\end{equation} Moreover, for any
$\pi\in\Rep(A;\mc{H})$ the quadruple
$(\pi(\ha),\pi(\hb),\pi(\hg),\pi(\hd))$ is a representation of the
Heisenberg-Lorentz commutation relations.
\end{thm}
\begin{proof}
We shall start by fixing some notation. Given any
$w\in\mb{C}\setminus\{0\}$, $g_w$ will denote an element of $G$ of the form:
\[g_w=\begin{pmatrix}
  0& w^{-1} \\
  -w & 0
\end{pmatrix}.\]
Let $T$ be a normal, invertible element acting on a Hilbert space
$\mc{H}$. We define the unitary operator:
\begin{equation}\label{sn}S(T)=\int dE^T(w)\ot
R_{g_w} \in\B(\mc{H}\ot L^2(G))\end{equation} where $E^T$ is the
spectral measure of $T$ and $R_{g_w}\in\B(L^2(G))$ is the right
shift by $g_w$. Let  $E^{R}$ be the  spectral measure that
corresponds to the representation $\Gamma\ni\gamma\mapsto
R_\gamma\in\B(L^2(G))$ via the S.N.A.G. theorem. For an
infinitesimal representation $(\tilde a,\tilde \lambda)$ of
the Heisenberg group $\mb{H}$ on a Hilbert space $\mc{H}$ we introduce the unitary
operator $R(\tilde a)\in\B(\mc{H}\ot L^2(G))$:
\begin{equation}\label{ra}R(\tilde a)=\int\, U^{\tilde a}_{-z,0}\ot
dE^{R}(z).
\end{equation}

Let us move on to the main part of the proof. An immediate
consequence of Definition \ref{repquar} is that it is enough to
consider two cases of represenations:
\begin{itemize}
\item[1.] $\tg=0$;
\item[2.] $\ker\tg=\{0\}$.
\end{itemize} We give the
proof for the second case leaving the first one to the reader.
Using Theorem \ref{thmc2} we define two closed operators
$\ta\tg^{-1}$ and $\td\tg^{-1}$ acting on $\mc{H}$. Note that
$(\ta\tg^{-1}, -s)$ and $(\td\tg^{-1},s)$ are infinitesimal
representations of $\mb{H}$. Using the notation introduced above
we define the unitary operator $U\in\B(\mc{H}\ot L^2(G))$:
\begin{equation}\label{urs}U=R(\td\tg^{-1})S(\tg)R(\ta\tg^{-1}).\end{equation}
Let us prove that U is a corepresentation of $\widehat{\mb{G}}$.
Let $\Hat\Delta\in\Mor(\C_r^*(G);\C_r^*(G)\ot\C_r^*(G))$ be the
canonical comultiplication on $\C_r^*(G)$. Note that
\begin{equation}\label{cs}
\begin{array}{rcl} (\id\ot\Hat\Delta)R(\td\tg^{-1})\hsp{=}\displaystyle\int\,
U^{\td\tg^{-1}}_{-(z+z'),0}\ot dE^{R}(z)\ot dE^{R}(z')\\\hsp{=}
\displaystyle\int\,\exp\left(-i\frac{s}{4}\Im(z\bar{z'})\right)U^{\td\tg^{-1}}_{-z,0}U^{\td\tg^{-1}}_{-z',0}
\ot dE^{R}(z)\ot
dE^{R}(z')\\\hsp{=}X^*_{23}R(\td\tg^{-1})_{12}R(\td\tg^{-1})_{13}.
\end{array}\end{equation}
The unitary element $X\in\M(\C_r^*(G)\ot\C_r^*(G))$ used above:
\begin{equation}\displaystyle X=\int\,\exp\left(i\frac{s}{4}\Im(z\bar{z}')\right)
dE^{R}(z)\ot dE^{R}(z')\end{equation}  is the one that twists
$\Hat\Delta$, giving the comultiplication $\Delta_{\widehat{\mb{G}}}$ (see \eqref{twcomul}).
Similarly, we check that
\begin{equation}\label{ct} (\id\ot\Hat\Delta)R(\ta\tg^{-1})=R(\ta\tg^{-1})_{12}R(\ta\tg^{-1})_{13}X_{23}.
\end{equation}
Moreover, the formula $\Hat\Delta(Z_w)=Z_w\ot Z_w$ implies that:
\begin{equation}\label{cr} (\id\ot\Hat\Delta)S(\tg)=S(\tg)_{12}S(\tg)_{13}.
\end{equation}
Using equations \eqref{cs}, \eqref{ct}, \eqref{cr}, the fact that
the first legs of $R(\ta\tg^{-1}),R(\td\tg^{-1})$ and $S(\tg)$
commute and  formula \eqref{twcomul} we get:
\begin{eqnarray*}
(\id\ot\Hat\Delta^\Psi)U\hsp{=}X_{23}(\id\ot\Hat\Delta)(R(\td\tg^{-1})S(\tg)R(\ta\tg^{-1}))X_{23}^*\\
\hsp{=}X_{23}X_{23}^*R(\td\tg^{-1})_{12}R(\td\tg^{-1})_{13}S(\tg)_{12}S(\tg)_{13}R(\ta\tg^{-1})_{12}R(\ta\tg^{-1})_{13}X_{23}X_{23}^*\\
\hsp{=}R(\td\tg^{-1})_{12}S(\tg)_{12}R(\ta\tg^{-1})_{12}R(\td\tg^{-1})_{13}S(\tg)_{13}R(\ta\tg^{-1})_{13}=U_{12}U_{13},
\end{eqnarray*}
which shows that $U$ is a corepresentation of $\widehat{\mb{G}}$. Let
$\pi_U\in\Rep(A;\mc{H})$ be the corresponding representation of
$A$. The next step is to prove that $\pi_U$ is the
representation $\pi$ of our theorem:
\[\pi_U(\ha)=\ta,\,\,\,\pi_U(\hb)=\tb,\,\,\,\pi_U(\hg)=\tg,\,\,\,\pi_U(\hd)=\td.\]
 Treating $\ha,\hg,\hd\in A^\eta$ as closed
operators acting on $L^2(G)$ (in particular $\hg$ is an invertible
operator of multiplication by the coordinate $\gamma$) one can
prove that the multiplicative unitary $\widehat{W}$ is given by
\begin{equation}\label{hw}\widehat{W}=R(\hd\hg^{-1})S(\hg)R(\ha\hg^{-1}).\end{equation}
It can also be proven that
 \begin{equation}\label{gam0}\begin{array}{rcl}
\widehat{W}^*(1\ot\exp(i\Im(z\hg))\widehat{W}\hsp{=}U^{\ha\ot\hg}_{z,0}U^{\hg\ot\hd}_{z,0},\\
U^*(1\ot\exp(i\Im(z\hg))U\hsp{=}U^{\ta\ot\hg}_{z,0}U^{\tg\ot\hd}_{z,0},\end{array}\end{equation}
which implies that
\begin{equation}\label{gam}\begin{array}{rcl}
\widehat{W}^*(1\ot\hg)\widehat{W}\hsp{=}\ha\ot\hg\dotplus\hg\ot\hd,\\
U^*(1\ot\hg)U\hsp{=}\ta\ot\hg\dotplus\tg\ot\hd.\end{array}\end{equation}
Applying $\pi_U\ot\id$ to both sides of the first of these
equations and using \eqref{corupi} we get
\begin{equation}\label{relpi}\pi_U(\ha)\ot\hg\dotplus\pi_U(\hg)\ot\hd
=\ta\ot\hg\dotplus\tg\ot\hd.\end{equation} Let
$\pi_0^\Psi\in\Mor(A;A_0)$ be the morphism introduced in Section
\ref{qbs}. It sends $\hg$ to $0$ and $\hd$ to the normal element
$\hd_0=\ha_0^{-1}\in A_0^\eta$. Applying $\id\ot\pi_0^\Psi$ to
both sides of \eqref{relpi} we get
\begin{equation}\label{tgth}
\pi_U(\hg)\ot\hd_0=\tg\ot\hd_0.
\end{equation}
This immediately implies that $\pi_U(\hg)=\tg$. From this equality
and \eqref{relpi} we see that $\pi_U(\ha)=\ta$.
 Now using \eqref{corupi} and \eqref{hw} we get
\[R(\pi_U(\hd)\tg^{-1})=R(\td\tg^{-1}).\]
Equation \eqref{ra}  together with the fact that the support
of the measure $dE^R$ is the whole complex plain implies that
$\pi_U(\hd)\tg^{-1}=\td\tg^{-1}$, hence $\pi_U(\hd)=\td$. Finally,
$\pi_U(\hb)=\tb$, which is a consequence of the related equalities
for $\ha,\hg$ and $\hd$.

The fact that for any representation $\pi\in\Rep(A;\mc{H})$ the
quadruple $(\pi(\ha),\pi(\hb),\pi(\hg),\pi(\hd))$ is a
representation of the Heisenberg-Lorentz commutation relations follows
directly from the definitions of affiliated elements
$\ha,\hb,\hg,\hd\in A^\eta$.
\end{proof}
The above theorem implies the following
\begin{cor}\label{sep} Let $A$ be the $\C^*$-algebra of the Heisenberg-Lorentz
quantum group. Then the generators $\ha,\hb,\hg,\hd\in A^\eta$ separate
representations of $A$. That is, if $\pi_1$ and $
\pi_2\in\Rep(A;\mc{H})$ coincide on $\ha,\hb,\hg,\hd$:
\[\pi_1(\ha)=\pi_2(\ha)\,,\, \pi_1(\hb)=\pi_2(\hb)\,,\,
\pi_1(\hg)=\pi_2(\hg)\,,\, \pi_1(\hd)=\pi_2(\hd)\] then
$\pi_1=\pi_2$.
\end{cor}
\end{subsection}
\begin{subsection}{$\ha,\hb,\hg,\hd$ as generators of  $A$}
By Corollary \ref{sep} we know  that
$\ha,\hb,\hg,\hd\in A^\eta$ separate representations of $A$. The
aim of this section is to prove that they generate $A$ in the
sense of Woronowicz. For this purpose we shall use the following
theorem which is a consequence of
 Theorem 4.2 of \cite{Worgen}.
\begin{thm}\label{congen}
Let $T_1,T_2\ldots,T_n$ be elements affiliated with a
$\C^*$-algebra $A$. Let $\Omega$ be the subset of $M(A)$
consisting of elements of the form $(1+T_i^*T_i)^{-1}$,
$(1+T_iT_i^*)^{-1}$, $\exp(-T_i^*T_i)$, $\exp(-T_iT_i^*)$. Assume
that
\begin{itemize}
\item[1.] $T_1,T_2\ldots,T_n$ separate representations;
\item[2.] there exist elements $r_1,r_2,\ldots, r_k\in \Omega$ such that  $r_1r_2\ldots r_k\in A$.
\end{itemize}
Then $T_1,T_2,\ldots,T_n$ generate $A$.
\end{thm}
\begin{thm}\label{gen} Affiliated elements $\ha,\hb,\hg,\hd$ generate the $\C^*$-algebra
$A$ of the Heisenberg-Lorentz quantum group. \end{thm}
\begin{proof}
From Theorem \ref{congen} and Corollary \ref{sep} we see that it
is enough to prove that
\begin{equation}\label{wprod}(1+\hb^*\hb)^{-1}\exp(-\ha^*\ha)\exp(-\hd^*\hd),\end{equation}
which is an element of $\M(A)$, belongs in fact to
$A\subset\M(A)$. In order to do that we shall first analyze the element $\exp(-\ha^*\ha)\exp(-\hd^*\hd)\in\M(A)$.
 For any $g\in G$ we set
\[f(g)=\frac{1}{
\cosh^2\left(\frac{s}{2}\gamma^*\gamma\right)}\exp\left(-\frac{2(|\alpha|^2+|\delta|^2)
\tanh\left(\frac{s}{2}\gamma^*\gamma\right)}{s\gamma^*\gamma}\right).\]
Let $h_t$ be the family of functions defined by \eqref{defh}. A
straightforward computation shows that
\begin{equation}\label{asa1}
\displaystyle f(g)=\int
d^2z_1\,d^2z_2\,h_{1}\left(z_1,-\frac{s}{2}|\gamma|^2\right)
h_{1}\left(z_2,\frac{s}{2}|\gamma|^2\right)\exp(i\Im(z_1\alpha))\exp(i\Im(z_2\delta)).
\end{equation}
One can check that the functions
$\exp(-|\gamma|^2)\exp(i\Im(z_1\alpha))$ and
$\exp(-|\gamma|^2)\exp(i\Im(z_2\delta))$ are quantizable in the
sense of Theorem \ref{A2} and
\begin{equation}\label{asa0}
\begin{array}{l}
\mc{Q}(\exp(-|\gamma|^2)\exp(i\Im(z_1\alpha)))=\exp(-|\gamma|^2)U_{z,0}^{\ha},\\
\mc{Q}(\exp(-|\gamma|^2)\exp(i\Im(z_2\delta)))=\exp(-|\gamma|^2)U_{z,0}^{\hd}.
\end{array}
\end{equation}
Using \eqref{asa1}, \eqref{asa0} and \eqref{asa}  we get
\begin{equation}\label{qf}\mc{Q}(f)=\exp(-\ha^*\ha)\exp(-\hd^*\hd).\end{equation}
Now for the purpose of analysis of the whole product \eqref{wprod}, we define two auxiliary functions
$k_1,k_2:G\rightarrow\mb{C}$:
\begin{equation}\label{qk}\begin{array}{rcl} k_1(g)\hsp{=}\displaystyle\frac{1}{1+\bar{\beta}\beta}f(g),\\
                  k_2(g)\hsp{=}f-(1+\Op(\beta)^*\Op(\beta))k_1.
\end{array}
\end{equation}
They satisfy the  assumptions of Theorem \ref{A1}, hence we can
quantize them obtaining $\mc{Q}(k_1)$ and $\mc{Q}(k_2)\in A$. Combining
\eqref{qf} and \eqref{qk} we see that
\begin{eqnarray*}(1+\hb^*\hb)^{-1}\exp(-\ha^*\ha)\exp(-\hd^*\hd)\hsp{=}(1+\hb^*\hb)^{-1}(\mc{Q}(f))\\
\hsp{=}(1+\hb^*\hb)^{-1}\mc{Q}((1+\Op(\beta)^*\Op(\beta))k_1+k_2)\\
\hsp{=}(1+\hb^*\hb)^{-1}\mc{Q}((1+\Op(\beta)^*\Op(\beta))k_1)+(1+\hb^*\hb)^{-1}\mc{Q}(k_2)\\
 \hsp{=}\mc{Q}(k_1)+(1+\hb^*\hb)^{-1}\mc{Q}(k_2).
\end{eqnarray*}
Both factors of the above sum belong to $A$, therefore
$(1+\hb^*\hb)^{-1}\exp(-\ha^*\ha)\exp(-\hd^*\hd)\in A$.
\end{proof}
\end{subsection}
\begin{subsection}{Comultiplication}
\begin{thm}\label{delg} Let $\mb{G}=(A,\Delta)$ be the Heisenberg-Lorentz quantum
group. Then  the action of $\Delta$ on the generators
$\ha,\hb,\hg,\hd\in A^\eta$ is given by
\begin{equation}\label{comg}\begin{array}{rcl}\Delta(\ha)\hsp{=}\ha\ot\ha\dotplus\hb\ot\hg\\
                                  \Delta(\hb)\hsp{=}\ha\ot\hb\dotplus\hb\ot\hd\\
                                  \Delta(\hg)\hsp{=}\hg\ot\ha\dotplus\hd\ot\hg\\
                                  \Delta(\hd)\hsp{=}\hg\ot\hb\dotplus\hd\ot\hd.
\end{array}
\end{equation}
\end{thm}
\begin{rem}\label{remcom} The elements appearing on the right hand side of \eqref{comg} (further denoted by $\ta,\tb,\tg,\td$ respectively) are treated as closed operators
acting on $L^2(G)\ot L^2(G)$. As usual, the sign $\dotplus$ denotes the
closure of the sum of two operators. The idea of the proof of the above theorem goes as follows. Using Theorem \ref{comrep} we see that the quadruple $(\ta,\tb,\tg,\td)$ is a representation of the Heisenberg-Lorentz commutation relations. By Theorem \ref{repth} this quadruple corresponds to a unique corepresentation $U$ of $\widehat{\mathbb{G}}$ on the Hilbert space $L^2(G)\otimes L^2(G)$ which in turn corresponds to a unique representation $\pi\in\Rep(A;L^2(G)\ot L^2(G))$
such that $\pi(\ha)=\ta,\,\pi(\hb)=\tb,\,\pi(\hg)=\tg,\,\pi(\hd)=\td$.
On the other hand $(\Delta\ot\id)\widehat{W}=\widehat{W}_{23}\widehat{W}_{13}$. Using the correspondence   $U=(\pi\ot\id)\widehat{W}$ and Theorem \ref{repth} once again we see that to prove \eqref{comg} it is enough to show that $U=\widehat{W}_{23}\widehat{W}_{13}$, which will be done in the following proof. Note that in order to prove Theorem \ref{delg}, it is first necessary to  show that the  quadruple $(\ta,\tb,\tg,\td)$ is a representation of the Heisenberg-Lorentz commutation relations. The proof of this fact seems to be as difficult as the proof of a more general Theorem \ref{comrep}.
\end{rem}
\begin{proof}[Proof of
Theorem \ref{delg}] In this proof we shall use the notation of the proof
of Theorem \ref{repth}.   Let $\ta,\tb,\tg,\td$  denote the right
hand sides of \eqref{comg}. As was explained in the above remark, to prove our theorem it is enough to show that
\begin{equation}\label{w0}\widehat{W}_{23}\widehat{W}_{13}=R(\td\tg^{-1})S(\tg)R(\ta\tg^{-1}).\end{equation}
 From equation \eqref{hw} we can see that
\begin{equation}\label{w1}\begin{array}{rcl}\widehat{W}_{13}\hsp{=}R(\hd\hg^{-1}\otimes
1)S(\hg\ot 1)R(\ha\hg^{-1}\ot 1),\\
\widehat{W}_{23}\hsp{=}R(1\ot\hd\hg^{-1})S(1\ot\hg)R(1\ot\ha\hg^{-1}).
\end{array}
\end{equation}
Therefore, the left hand side of \eqref{w0} has the following
form:
\begin{equation}\label{w2}R(1\ot\hd\hg^{-1})S(1\ot\hg)R(1\ot\ha\hg^{-1})R(\hd\hg^{-1}\otimes
1)S(\hg\ot 1)R(\ha\hg^{-1}\ot 1).\end{equation} Using the fact
that $R(1\ot\ha\hg^{-1})$ commutes with $R(\hd\hg^{-1}\otimes 1)$
we see that \eqref{w2} is equal to
\begin{equation}\label{w3}R(1\ot\hd\hg^{-1})S(1\ot\hg)R(\hd\hg^{-1}\ot 1)
R(1\ot\ha\hg^{-1})S(\hg\ot 1)R(\ha\hg^{-1}\ot 1).\end{equation}
Formula \eqref{formal} and the corresponding formula  related to $\td$
imply that
\begin{equation}\label{w4}\begin{array}{rcl}
\ha\hg^{-1}\ot 1\hsp{=}\ta\tg^{-1}\dotplus\tg^{-1}(\hg^{-1}\ot\hg),\\
1\ot\hd\hg^{-1}\hsp{=}\td\tg^{-1}\dotplus\tg^{-1}(\hg\ot\hg^{-1}).
\end{array}
\end{equation}
Using these equalities and the fact that \eqref{w2} is equal to
$\widehat{W}_{23}\widehat{W}_{13}$ we get:
\begin{equation}\label{w5}
\begin{array}{rcl}
\widehat{W}_{23}\widehat{W}_{13}\hsp{=}R(\td\tg^{-1})\exp(-i\Im(\tg^{-1}(\hg\ot\hg^{-1})\ot
T_r))S(1\ot\hg)R(\hd\hg^{-1}\ot 1)\\\hsp{\times}
R(1\ot\ha\hg^{-1})S(\hg\ot 1)\exp(-i\Im(\tg^{-1}(\hg^{-1}\ot\hg)\ot
T_r))R(\ta\tg^{-1}).
\end{array}
\end{equation}
Noting that \[R(\hd\hg^{-1}\ot 1)
R(1\ot\ha\hg^{-1})=\exp(-i\Im(\tg(\hg^{-1}\ot\hg^{-1})\ot T_r))\]
and using equation \eqref{w5} we see that in order to prove
equality \eqref{w0} it is enough to check that
\begin{equation}\label{w6}\begin{array}{rcl}S(\tg)\hsp{=}\exp(-i\Im(\tg^{-1}(\hg\ot\hg^{-1})\ot
T_r))S(1\ot\hg)\\\hsp{\times}\exp(-i\Im(\tg(\hg^{-1}\ot\hg^{-1})\ot
T_r))S(\hg\ot 1)\exp(-i\Im(\tg^{-1}(\hg^{-1}\ot\hg)\ot
T_r)).\end{array}\end{equation} The operators $\tg,1\ot\hg$ and $\hg\ot 1$
which appear in the above expression are normal and they strongly
commute. Therefore, to prove \eqref{w6} we have to check that
\begin{equation}\label{w7}\begin{array}{rcl}S(u)\hsp{=}\exp(-i\Im(u^{-1}vw^{-1}
T_r))S(w)\exp(-i\Im(uv^{-1}w^{-1}
T_r))S(v)\exp(-i\Im(u^{-1}v^{-1}wT_r))\end{array}\end{equation}
for any $u,v,w\in\mb{C}\setminus\{0\}$. Noting that
\begin{equation}\label{sz} S(w)=Z_w,\,\,\, \exp(i\Im(zT_r))=R_z,\end{equation} where $Z_w$ and
$R_z$ are operators defined in the proof of Theorem \ref{repth}, we see
that equation \eqref{w7} is equivalent to the following matrix
identity:
\[\begin{array}{rcl}
\left(
  \begin{array}{cc}
    0 & u^{-1} \\
    -u & 0 \\
  \end{array}
\right)\hsp{=}\left(
                \begin{array}{cc}
                  1 & -vu^{-1}w^{-1} \\
                  0 & 1 \\
                \end{array}
              \right)\left(
                       \begin{array}{cc}
                         0 & w^{-1} \\
                         -w & 0 \\
                       \end{array}
                     \right)\left(
                              \begin{array}{cc}
                                1 & -uv^{-1}w^{-1} \\
                                0 & 1 \\
                              \end{array}
                            \right)\vspace{0,5cm}\\\hsp{\times}\left(
                                                   \begin{array}{cc}
                                                     0 & v^{-1} \\
                                                     -v & 0 \\
                                                   \end{array}
                                                 \right)\left(
                                                          \begin{array}{cc}
                                                            1 & -wu^{-1}v^{-1} \\
                                                            0 & 1 \\
                                                          \end{array}
                                                        \right).
\end{array}\]
Its verification is a straightforward computation which is left to
the reader.
\end{proof}
\end{subsection}
\end{section}
\begin{appendices}
\begin{section}{Quantization map}\label{Aqm}
 Let $\mathcal{Q}$ be the quantization map introduced in
Definition 4.13 of paper \cite{Kasp}. Recall that $\mc{Q}$ was
defined on elements of the Fourier algebra $\mc{F}$: \[
\mc{F}=\{(\omega\otimes \id)V:\omega\in\B(L^2(G))_*\}\] and
$\mathcal{Q}((\omega\otimes \id)V)=(\omega\otimes \id)W$, where
$V$ is the Kac-Takesaki operator of a locally compact group $G$
and $W$ is the multiplicative unitary related to a Rieffel
deformation of $G$. Given a function $f\in\C_0(G)$ it is usually
difficult to check if $f\in\mc{F}$, which makes $\mc{Q}$ not very
useful in the practical applications. In the case of the Heisenberg-Lorentz quantum group we
shall give a new description of the quantization map which does
not have the aforementioned drawback. $\mc{Q}$ will be defined
on a different class of functions but when the function happens to
be an element of $\mc{F}$ then the new definition will coincide
with the old one.
 Consider two representations of
$\Gamma\subset G$ on $L^2(G)$:
\begin{equation}\label{lrrep}
\begin{array}{c}
\Gamma\ni\gamma\mapsto R_\gamma\in\B(L^2(G)),\\
\Gamma\ni\gamma\mapsto L_\gamma\in\B(L^2(G)).
\end{array}
\end{equation}
Let  $T_l$ and $T_r$ be infinitesimal generators of these
representations:
\begin{equation}\label{tltr}\begin{array}{rcl}R_\gamma\hsp{=}\exp(i\Im(\gamma T_r)),\\
L_\gamma\hsp{=}\exp(i\Im(\gamma T_l))\end{array}\end{equation} for
any $\gamma\in \Gamma\simeq\mathbb{C}$. The related complex vector fields on $G$
will be denoted by $\partial_l$ and $\partial_r$ (see equation
\eqref{diffop1}). Now consider two differential operators
$K_l=(1+T_l^*T_l)^{2}$ and $K_r=(1+T_r^*T_r)^{2}$ acting on
$L^2(G)$. Note that $K_l$ and $K_r$ are positive, invertible and
their inverses $K_l^{-1}, K_r^{-1}$ are bounded.

 Let $x,y,v,w\in L^2(G)$ be vectors such that
$x,y,w\in\D(K_r)$ and $v\in\D(K_l)$. Our next objective is to
compute the matrix element $\me{x\ot v}{W}{y\ot w}$. Note that
\begin{eqnarray}\label{mew}\me{x\ot v}{W}{y\ot w}\hsp{=}
\me{K_r x\ot K_l v}{(K_r^{-1}\ot K_l^{-1})YVX (K_r^{-1}\ot
K_r^{-1})}{K_r y\ot K_r w}.\end{eqnarray} Let
$\pi_R,\pi_L\in\Rep(\C_0(\mb{C});L^2(G))$
 be representations of
$\C_0(\mb{C})$ which  send $\id\in\C_0(\mb{C})^\eta$ to $T_r$ and
$T_l$ respectively. We have the following two equalities:
\begin{eqnarray}\label{xfor} X(K_r^{-1}\ot K_r^{-1})\hsp{=}(\pi_R\ot
\pi_R)((K^{-1}\ot K^{-1})\Psi)\\\label{yfor}(K_r^{-1}\ot
K_l^{-1})Y\hsp{=}(\pi_R\ot \pi_L)((K^{-1}\ot
K^{-1})\Psi),\end{eqnarray} where $K:\mb{C}\rightarrow \mb{R}$ is
the function given by the formula: \[K(z)=(1+|z|^2)^2\] and
$\Psi\in\M(\C_0(\mb{C})\otimes\C_0(\mb{C}))$ is defined by
$\eqref{dualbi}$. Let $l:\mb{C}^2\rightarrow\mb{C}$ be the
function given by:
\[l(w_1,w_2)=\int
d^2z_1d^2z_2\,\frac{\exp(-i\Im(z_1w_1+z_2w_2-\frac{s}{4}z_1\bar{z}_2))}{(1+|z_1|^2)^{2}(1+|z_2|^2)^{2}}.\]
Note that  $l\in L^1(\mb{C}^2)$ and the right hand side of
$\eqref{xfor}$ can be expressed by $l$:
\begin{equation}\label{xrfor}(\pi_R\ot
\pi_R)((K^{-1}\ot K^{-1})\Psi)=\int
d^2w_1d^2w_2\,l(w_1,w_2)(R_{w_1}\ot R_{w_2}).\end{equation} We
have a similar formula for the right hand side of \eqref{yfor}:
\begin{equation}\label{yrfor}(\pi_R\ot \pi_L)((K^{-1}\ot
K^{-1})\Psi)=\int d^2w_1d^2w_2\,l(w_1,w_2)(R_{w_1}\ot
L_{w_2}).\end{equation} Let $f=(\omega_{x,y}\ot\id)V\in\C_0(G)$ be
the slice of the Kac-Takesaki operator and
$\tilde{h}\in\C_0(\mb{C})$ the function given by:
\begin{equation}\label{th}\tilde{h}(w)=\displaystyle\int d^2z \frac{\exp(
i\Im(wz))}{(1+s^{-2}|z|^2)^2}.\end{equation}  Let $\pi^{\rm
can}\in\Rep(\C_0(G)\rtimes\mb{C}^2;L^2(G))$ be  the
representation introduced in Remark 4.5 of \cite{Kasp} and
$\lambda^L,\lambda^R\in\Mor(\C_0(\mb{C});\C_0(G)\rtimes\mb{C}^2)$
 the morphisms introduced in the paragraph following
Proposition 4.2 of \cite{Kasp}. A simple but tedious computation
which starts with inserting \eqref{xrfor} and \eqref{yrfor} into
\eqref{mew} leads to the following equality:
\begin{eqnarray*}\label{mew1}\me{x\ot v}{W}{y\ot
w}&\\&\hspace*{-3cm}=\displaystyle\int d^2w_1d^2w_2\me{v}{\pi^{\rm
can}\Bigr(\hat{\rho}_{w_1,w_2}^{\Psi}\bigl
(\lambda^L(\tilde{h})\bigl((1+\partial_r^*\partial_r)^2(1+\partial_l^*\partial_l)^2f\bigr)
\lambda^R(\tilde{h})\bigr)\Bigl)}{w}\end{eqnarray*}  Denoting
\[\lambda^L(\tilde{h})\bigl((1+\partial_r^*\partial_r)^2(1+\partial_l^*\partial_l)^2f\bigr)
\lambda^R(\tilde{h})\in\M(\C_0(G)\rtimes\mb{C}^2)\]  by $b_f$ we
get \begin{equation}\label{mew2}\me{x\ot v}{W}{y\ot
w}=\displaystyle\int d^2w_1d^2w_2\me{v}{\pi^{\rm
can}(\hat{\rho}_{w_1,w_2}^{\Psi}(b_f))}{w}.\end{equation} If $b_f$
happens to be in the domain of $\mathfrak{E}^\Psi$ - the averaging
map
 with respect to the twisted dual action
$\hat\rho^\Psi$ - then, in the formula \eqref{mew2} we can enter
the integral under the scalar product obtaining
\begin{equation}\label{mew3}\me{x\ot v}{W}{y\ot w}=\me{v}{\pi^{\rm can}\left(\int
d^2w_1d^2w_2\,\hat{\rho}_{w_1,w_2}^{\Psi}(b_f)\right)}{w}.\end{equation}
Equation \eqref{mew3} may then be rewritten as follows:
\[\mc{Q}(f)=\pi^{\rm can}(\mathfrak{E}^\Psi(b_f)).\] Let us show
that this last equation holds whenever $f$ is regular enough. In
the next theorem we shall keep the same notation  $T_l$ and $T_r$
for normal operators acting on $L^2(G)$ (see \eqref{tltr}) and
elements affiliated with $\C_0(G)\rtimes\mb{C}^2$ (see
\eqref{infgenl}).
\begin{thm} \label{A1}Let $\partial_l,\partial_r$ be the complex vector
fields on $G$ given by \eqref{diffop1} and $f\in\C_0(G)$ a
continuous function such that\,\,\,
$\partial_l^{k*}\partial_l^{k'}\partial_r^{*m}\partial_r^{m'}f\in\C_0(G)$\,\,\,
whenever $k,k',l,l'\leq 5$. Let $b_f\in\M(\C_0(G)\rtimes\mb{C}^2)$
be the element given by
\[b_f=\lambda^L(\tilde{h})\bigl((1+\partial_r^*\partial_r)^2(1+\partial_l^*\partial_l)^2f\bigr)
\lambda^R(\tilde{h}).\] Then $b_f$  is in the domain of the
averaging map $\mathfrak{E}^\Psi$ and there exists a positive
constant $c$ such that
\[\|\mathfrak{E}^\Psi(b_f)\|\leq c\max_{k,k',l,l'\leq 5}
\sup_{g\in G}\left|\partial_l^{k*}\partial_l^{k'}\partial_r^{*m}\partial_r^{m'}f\right|.\]
 If  $f\in\C_0(G)$ is quantizable in the sense of Definition 4.13 of
 \cite{Kasp}, then
\[\mc{Q}(f)=\pi^{\rm can}(\mathfrak{E}^\Psi(f)).\]
\end{thm}
To prove the above theorem we shall need the following lemma:
\begin{lem}\label{lema2}
Let $X$ be a locally compact Hausdorff space,
$\rho:\mb{C}\rightarrow\Aut(\C_0(X))$  a continuous action and
$(\C_0(X)\rtimes\mb{C},\lambda,\hat\rho)$  the canonical
$\mb{C}$-product associated with $\rho$. Let $\partial$,
$\partial^*$ be the differential operators acting on the smooth
domain $\D^\infty(\rho)\subset\C_0(X) $ of the action $\rho$:
\begin{equation}
\begin{array}{rcl}\partial f\hsp{=}\frac{\partial}{\partial z}\rho_zf|_{z=0}\vspace{.1cm}\\
\partial^* f\hsp{=}\frac{\partial}{\partial \bar
z}\rho_zf|_{z=0}\end{array}.
\end{equation} Further, let $\tilde{h}\in\C_0(\mb{C})$ be the function
defined by \eqref{th} and let $g\in\C_0(X)$ be such that\,\,\,
$\partial^{*l}\partial^{k}g\in\C_0(\mb{C})$ for $k,l\in\{0,1\}$.
Then $\lambda(\tilde{h})g$ is in the domain of the averaging map
$\mathfrak{E}$ and there exists a positive constant $c\in\mb{R}$
such that:
\begin{equation}\label{et1}\|\mathfrak{E}(\lambda(\tilde{h})g)\|\leq c\max_{l,k\leq 1}\sup_{x\in
X}|\partial^{*l}\partial^{k}g(x)|.\end{equation}
\end{lem}
\begin{proof}
Using the universal properties of the group $\C^*$-algebra
$\C^*(\mathbb{C})$ we see that the representation
$\lambda\in\Rep(\mathbb{C};\C_0(X)\rtimes\mb{C})$ corresponds to a
unique element of
 $\Mor(\C^*(\mathbb{C});\C_0(X)\rtimes\mb{C})$ (which we also denote by $\lambda$). Identifying
$\C^*(\mathbb{C})$ with $\C_0(\mathbb{C})$ (note that we use the self-duality of $\mb{C}$) we  can apply  $\lambda$
to $\tilde h\in \C_0(\mathbb{C})$: $\lambda(\tilde
h)\in\M(\C_0(X)\rtimes\mb{C})$.

In order to show that  $\lambda(\tilde{h})g$ is in the domain of
the averaging map $\D(\mathfrak{E})$ it is enough to express  it
as a linear combination of elements of the form
\begin{equation}\label{e0}\lambda(h_1)b\lambda(h_2)\end{equation} where
$h_1,h_2\in\C_0(\mb{C})\cap L^2(\mb{C})$ and
$b\in\C_0(X)\rtimes\mb{C}$ (see \cite{Ped}).  Let
$T\,\eta\,\C_0(X)\rtimes\mb{C}$ be the image of
$\id\in\C_0(\mathbb{C})$ under
$\lambda\in\Mor(\C_0(\mathbb{C});\C_0(X)\rtimes\mb{C})$:
$T=\lambda(\id)$. Note that
\begin{equation}\label{e1}
\begin{array}{rcl}
\lambda(\tilde{h})g\hsp{=}\lambda(\tilde{h})g(1+T^*T)(1+T^*T)^{-1}\\
\hsp{=}\lambda(\tilde{h})g(1+T^*T)^{-1}+\lambda(\tilde{h}\bar{z})gT(1+T^*T)^{-1}
+\lambda(\tilde{h})\partial^*gT(1+T^*T)^{-1}
\end{array}
\end{equation} where we used the relation linking $\partial^*$ and
$T^*$:
 \[\partial^* g=[g,T^*].\] Note also that
$(1+T^*T)^{-1}=\lambda((1+|z|^2)^{-1})$, hence
\[\lambda(\tilde{h})=\lambda((1+|z|^2)\tilde{h})(1+T^*T)^{-1}.\]
Therefore, the first summand of the right hand side  of \eqref{e1}
is of the form
\[\lambda(\tilde{h})g(1+T^*T)^{-1}=\lambda((1+|z|^2)\tilde{h})((1+T^*T)^{-1}g)\lambda((1+|z|^2)^{-1}).\]
Using the fact that $\tilde{h}$ is of the Schwartz type and $(1+|z|^2)^{-1}\in
L^2(\mb{C})$ we can see
that the above element is of the form \eqref{e0}. Now, by
inequality (10) of paper \cite{Kasp} we get
\[
\|\mathfrak{E}(\lambda(\tilde{h})g(1+T^*T)^{-1})\|\leq\|\tilde{h}\|_2\,\|g\|\,\|(1+|z|^2)^{-1}\|_2
\]
and see that $\|\mathfrak{E}(\lambda(\tilde{h})g(1+T^*T)^{-1})\|$
may be estimated by the right hand side of \eqref{et1} for $c'$
big enough:
\begin{equation}\label{e3}
\|\mathfrak{E}(\lambda(\tilde{h})g(1+T^*T)^{-1})\|\leq
c'\max_{l,k\leq 1}\sup_{x\in X}|\partial^{*l}\partial^{k}g(x)|
\end{equation}
 Let us analyze the second summand of the right hand side of
\eqref{e1}. Note that
\[\lambda(\tilde{h}\bar{z})gT(1+T^*T)^{-1}=\lambda(\tilde{h}|z|^2)g(1+T^*T)^{-1}
+\lambda(\tilde{h}\bar{z})\partial g(1+T^*T)^{-1}.\] A reasoning
similar to the one above shows that there exists a constant $c''$
such that
\begin{equation}\label{e4}
\|\mathfrak{E}(\lambda(\tilde{h}\bar{z})gT(1+T^*T)^{-1})\|\leq
c''\max_{l,k\leq 1}\sup_{x\in X}|\partial^{*l}\partial^{k}g(x)|.
\end{equation}
Similarly, we prove that there exists a constant  $c'''$ such that
\begin{equation}\label{e5}
\|\mathfrak{E}(\lambda(\tilde{h})\partial^*gT(1+T^*T)^{-1})\|\leq
c'''\max_{l,k\leq 1}\sup_{x\in X}|\partial^{*l}\partial^{k}g(x)|.
\end{equation}
Combining \eqref{e1}, \eqref{e3}, \eqref{e4} and \eqref{e5} we get
\eqref{et1} for $c=\max\{c',c'',c'''\}$.
\end{proof}
The above lemma is also true if we replace
$\mathfrak{E}$ with $\mathfrak{E}^\Psi$. An extension of this
lemma to the case of an action of $\mb{C}^2$ gives a proof of
 Theorem \ref{A1}. Using the same techniques  one can
also prove the following theorem:
\begin{thm} \label{A2}Let $f\in\C_{\rm b}(G)$ be a function such
that
\,\,$\partial_l^{k*}\partial_l^{k'}\partial_r^{*m}\partial_r^{m'}f\in\C_{\rm
b}(G)$\,\, whenever $k,k',l,l'\leq 5$. Let
$b_f\in\M(\C_0(G)\rtimes\mb{C}^2)$ be  given by:
\[b_f=\lambda^L(\tilde{h})\bigl((1+\partial_r^*\partial_r)^2(1+\partial_l^*\partial_l)^2f\bigr)
\lambda^R(\tilde{h}).\] Then $b_f\in\D(\mathfrak{E}^\Psi)$,
$\mathfrak{E}^\Psi(b_f)\in\M(A)$ and there exists a positive
constant $c$ such that
\[\|\mathfrak{E}^\Psi(b_f)\|\leq c\max_{k,k',l,l'\leq 5}
\sup_{g\in
G}\left|\partial_l^{k*}\partial_l^{k'}\partial_r^{*m}\partial_r^{m'}f\right|.\]
\end{thm}
The element  $\mathfrak{E}^\Psi(b_f)\in\M(A)$ appearing in the
 above theorem will also be denoted by $\mc{Q}(f)$.
\end{section}
\begin{section}{Counit in Rieffel
deformation}\label{ApD} The aim of this section is to show that a
quantum group $\mb{G}$ obtained by the M. Rieffel deformation
posses a counit. Our argument will be different from the one
given by Rieffel in \cite{Rf2}. Let $G$ be a locally compact
group, $\Gamma\subset G$  its abelian subgroup and $\Psi$ a
$2$-cocycle on $\Hat\Gamma$. Let $\rho$ be the action of
$\Gamma^2$ on $\C_0(G)$ given by the left and right shifts and let
$\rho_\Gamma$ be the corresponding action of $\Gamma^2$ on
$\C_0(\Gamma)$. Note that the restriction morphism
$\pi_\Gamma:\C_0(G)\rightarrow \C_0(\Gamma)$ is
$\Gamma^2$-covariant. Using Proposition 3.8 of \cite{Kasp} we get
 the induced morphism
$\pi_\Gamma^\Psi\in\Mor(\C_0(G)^\Psi;\C_0(\Gamma)^\Psi)$. The
abelianity of $\Gamma$ implies that the dual quantum group of
$(\C_0(\Gamma)^\Psi,\Delta)$ coincides with the quantum group
$(\C^*(\Gamma),\Hat\Delta)$. Therefore
$(\C_0(\Gamma)^\Psi,\Delta)$ coincides with
$(\C_0(\Gamma),\Delta)$. This shows that
$\pi_\Gamma^\Psi\in\Mor(\C_0(G)^\Psi;\C_0(\Gamma))$ and enables us
to define a counit $e$ for $\mb{G}$ by the formula
$e(a)=e_\Gamma(\pi_\Gamma^\Psi(a))$ for any $a\in A$, where
$e_\Gamma:\C_0(\Gamma)\rightarrow\mb{C}$ is the counit for
$(\C_0(\Gamma),\Delta)$.

Let us now draw an important conclusion from the existence of the
counit for $\mb{G}$. Using Proposition 5.16 of \cite{SolWor} we can see
that the universal dual quantum group of
$\widehat{\mb{G}}=(\C_r^*(G),\Delta^\Psi)$ is isomorphic with the
reduced dual: $\mb{G}=(A,\Delta)$. In particular, representations
of $\C^*$-algebra $A$ are in one to one correspondence with
corepresentations of quantum group $\widehat{\mb{G}}$. This follows from
Theorem 5.4 of \cite{SolWor}.
\end{section}
\begin{section}{Complex generator of Heisenberg Lie
algebra}\label{bcg} Let $\mb{H}$ be the Heisenberg group,
$\mathfrak{h}$ its Lie algebra and $\mc{E}$ the enveloping algebra
of $\mathfrak{h}$. $\mc{E}$ is generated by an element
$a\in\mc{E}$ such that the commutator $\lambda=[a^*,a]$ is central
in $\mc{E}$. Let $A$ be a $\C^*$-algebra and let
$U\in\Rep(\mb{H};A)$ be a representation. As was described in the
third chapter of \cite{WorNap}, $U$ induces the map
\[dU:\mc{E}\rightarrow\{\mbox{closed maps on } A\}.\]
By $\D^\infty(U)$ we shall denote the set of $U$-smooth elements
in $A$. In the next definition we shall identify a representation
of $\mb{H}$ in  the $\C^*$-algebra of compact operators
$\mc{K}(\mc{H})$ with the corresponding Hilbert space
representation.

\begin{defin}\label{c0} Let $\mc{H}$ be a Hilbert space and let
$(\tilde{a},\tilde{\lambda})$ be a pair of closed operators acting
on $\mc{H}$. We say that this pair is an infinitesimal
representation of $\mb{H}$ on $\mc{H}$ if there exists a
representation $U\in\Rep(\mb{H};\mc{K}(\mc{H}))$ such that
$dU(a)=\tilde a$ and $dU(\lambda)=\tilde\lambda$.
\end{defin}
The representation $U$ in the above definition is determined by
$\tilde a$, therefore in this context  it will be denoted by
$U^{\tilde a}$. Let  $U\in\Rep(\mb{H};\C^*(\mb{H}))$ be the
canonical representation of $\mb{H}$. The map $dU$ in this case is
injective, which enables us to identify $dU(T)$ with $T\in\mc{E}$.
The aim of this section is to show that $a\in\mc{E}$ is affiliated
with $\C^*(\mb{H})$. In fact one can
 prove that $a$ generates $\C^*(\mb{H})$ in the sense of
Woronowicz but we shall not use and so will not prove this fact. 

Let
$M\in\mc{E}$. The criterion for a map\,\, $dU(M):\D(dU(M))\rightarrow
\C^*(\mb{H})$\,\, to be affiliated with $\C^*(\mb{H})$ is provided by
 Theorem 2.1 of \cite{WorNap}. Our proof that
$a\,\eta\C^*(\mb{H})$  uses a different approach which is based
on the explicit construction of the semigroup\,\, $\displaystyle
\mb{R}_+\ni t\mapsto \exp(-ta^*a)\in\M(\C^*(\mb{H}))$.
\begin{thm} Let $a$ be the complex generator  of the
algebra $\mc{E}$. Then $a$ is affiliated with $\C^*(\mb{H})$.
\end{thm}
\begin{proof}
For any $z\in\mathbb{C}$, $x\in\mathbb{R}$ and $t\in\mathbb{R}_+$
we set
\begin{equation}\label{defh}h_t(z,x)=\frac{x\exp tx}{4\pi\sinh
tx}\exp\left(-\frac{|z|^2x\coth tx}{4}\right)\in
\mathbb{R}_+.\end{equation} We would like to define an element
$H_t\in\M(\C^*(\mb{H}))$ by the integral
\[H_t=\int_{\mb{C}} d^2z\,h_t\left(z,\frac{1}{2}\lambda\right)U_{z,0}\]
but there is a problem with its convergence. To circumvent it we observe that
 for any $b\in\C^*(\mb{H})$ and
 $f\in\C^\infty_c(\mb{C})$ the
 integral
 \[\int d^2z\,h_t\left(z,\frac{1}{2}\lambda\right)U_{z,0}f(\lambda)b\]  converges in the  norm sense
and the following inequality holds:
 \[\left\|\int
 d^2z\,h_t\left(z,\frac{1}{2}\lambda\right)U_{z,0}f(\lambda)b\right\|\leq
 \Big\|f(\lambda)b\Big\|.\]
 Hence $H_t$ is well defined on the elements of the form
 $f(\lambda)b$ and by the above inequality it can be extended
 to the whole $\C^*(\mb{H})$ giving a selfadjoint element of $\M(\C^*(\mb{H}))$.
Let us list some properties of $H_t$.
 \begin{itemize}
 \item[1.] The map\,\, $\mb{R}_+\ni t\mapsto H_t\in\M(\C^*(\mb{H}))$\,\,
 is a norm-continuous semigroup  and $\|H_t\|\leq 1$.
 \item[2.] $\displaystyle\lim_{t\rightarrow 0}H_tb=b$ for any
 $b\in\C^*(\mb{H})$.
 \item[3.] For any $b\in\D^\infty(U)$ the
 map $\mb{R}_+\ni t\mapsto H_tb$ is differentiable and
 \[\left.\frac{d}{dt}H_tb\right|_{t=0}=-a^*a\,b.\]
 \end{itemize}
Point 1 of this list enables us to define the element
$\Xi\in\M(\C^*(\mb{H}))$:
 \[\Xi=\int_{\mb{R}_+}dt\, e^{-t}H_t.\] Using point 2 and 3 we can check that
 for any $b\in\D^\infty(U)$ we have
 $\Xi b\in\D^\infty(U)$ and \[(1+a^*a)\Xi b=b.\] This shows that
 $\overline{(1+a^*a)\D^\infty(U)}^{\|\cdot\|}= \C^*(\mb{H})$, which by Proposition 2.2 of \cite{Worun} is
 sufficient for $a$ to be affiliated with
 $\C^*(\mb{H})$.
\end{proof}
\begin{rem} Let $\mc{H}$ be a Hilbert space.
 Analyzing the above proof one can conclude that given
any representation $\pi\in\Rep(\C^*(\mb{H});\mc{H})$, a compactly
supported function $f\in\C_0(\mb{C})$ and $v\in\mc{H}$ we have
\begin{equation}\displaystyle\label{asa}\exp(-t\pi(a)^*\pi(a))\pi(f(\lambda))v=
\int_{\mb{C}}
d^2z\,h_t\left(z,\frac{1}{2}\pi(\lambda)\right)\pi(U_{z,0}f(\lambda))v\end{equation}
where the integral on the right is taken in the sense of norm
topology on $\mc{H}$. Let us also note that given any $v\in\mc{H}$
such that the differential
\[\frac{\partial}{\partial z}\pi(U_{z,0})v\Big|_{z=0}\] exists, we
have $v\in\D(\pi(a))$ and
\begin{equation}\label{remdif}
\pi(a)h=2\frac{\partial}{\partial z}\pi(U_{z,0})v\Big|_{z=0}.
\end{equation}
Further, let $B$ be a $\C^*$-algebra  and
$\pi\in\Mor(\C^*(\mb{H});B)$. For any $b\in B$, such that the
differential
\[\frac{\partial}{\partial z}\pi(U_{z,0})b\Big|_{z=0}\] exists,
we have $b\in\D(\pi(a))$ and
\begin{equation}\label{remdif1}
\pi(a)b=2\frac{\partial}{\partial z}\pi(U_{z,0})b\Big|_{z=0}.
\end{equation}
\end{rem}
\end{section}
\begin{section}{Product of affiliated elements}
Let $A$ be a $\C^*$-algebra, and $T_1,T_2\,\eta\, A$. In general,
 the product of $T_1$ and $T_2$ is not well defined,
  but it can be defined,  assuming that $T_1$ and $T_2$ commute in a good sense.
  The construction of the product given here is a generalization
of the case when $A=A_1\ot A_2$, $T_1=S_1\ot 1$ and $T_2=1\ot S_2$,
where $S_1\,\eta\,A_1$ and $S_2\,\eta\,A_2$. Then the product
of $T_1$ and $T_2$ is the tensor product $S_1\ot
S_2\,\eta\,A_1\ot A_2$ whose construction was  described in \cite{WorNap}.
\begin{defin}\label{c1}
Let $A$ be a $\C^*$-algebra and let $T_1,T_2$ be elements
affiliated with $A$. Let $z_1,z_2\in\M(A)$ be $z$-transforms of
$T_1$ and $T_2$ respectively. We say that $T_1$ and $T_2$ strongly
commute if
\begin{eqnarray}\label{zjd}z_1z_2\hsp{=}z_2z_1,\\
\label{zjdg} z_1^*z_2\hsp{=}z_2z_1^*.
\end{eqnarray}
Let $T_1$ and $T_2$ be a pair of closed operators acting on a
Hilbert space $\mc{H}$. We say that $T_1$ and $T_2$ strongly
commute if they  strongly commute as elements affiliated with the
algebra of compact operators $\mc{K}(\mc{H})$.
\end{defin}
\begin{thm}\label{thmc2} Let $A$ be a $\C^*$-algebra and let $T_1,T_2\,\eta\, A$ be a
strongly commuting pair of affiliated elements. Let us consider
the set $\D(T_0)=\{a\in\D(T_2):T_2a\in\D(T_1)\}$ and define an
operator $T_0:\D(T_0)\rightarrow A$ by the formula
$T_0a=T_1(T_2a)$. Then $T_0$ a is closable operator acting  on the
Banach space $A$ and its closure $T^{\rm cl}_0$ is affiliated with
$A$. This closure will be denoted by $T_1T_2$. We also have
$T_1T_2=T_2T_1$.
\end{thm}
\begin{proof}
We define $T_1T_2$ using the method described in Theorem 2.3 of
\cite{Worun}. The related matrix $Q\in\M(A)\ot\M(\mb{C}^2)$ has
the form:
\[Q=\begin{pmatrix}
  (1-z_1^*z_1)^{\frac{1}{2}}(1-z_2^*z_2)^{\frac{1}{2}} & -z_1^*z_2^* \\
  z_1z_2 & (1-z_1z_1^*)^{\frac{1}{2}}(1-z_2z_2^*)^{\frac{1}{2}}
\end{pmatrix}.\] (compare with the  matrix $Q$ from the
proof of Theorem 6.1 of \cite{WorNap}). $Q$ satisfies all the
assumptions of Theorem 2.3, hence it gives rise to an affiliated
element. We leave it to the reader to check that this affiliated
element is $T\in A^\eta$ of our theorem.
\end{proof}
For the needs of this paper we shall prove the following lemmas.
\begin{lem} \label{lemcore}Let $A$ be a $\C^*$-algebra, $T$
an element affiliated with $A$ and  $X$  a dense subspace of
$\D(T)$. Then :
\begin{itemize}
\item[(1)] if
$(1+T^*T)^{\frac{1}{2}}X$ is dense in $A$, then $X$ is a core of
$T$;
\item[(2)] if $X\subset\D(T^*T)$ and $(1+T^*T)X$ is dense in $A$,
then $X$ is a core of $T$.
\end{itemize}
\end{lem}
\begin{proof} It is easy to see that for any dense subspace $X'\subset
A$ the set $(1+T^*T)^{-\frac{1}{2}}X'$ is a core of $T$. Taking
$X'=(1+T^*T)^{\frac{1}{2}}X$ we get the proof of  point (1) of our
lemma. To prove point (2) note that $(1+T^*T)^{-\frac{1}{2}}X'$ is
dense in $A$ whenever $X'$ is dense in $A$. Applying this to the
set $(1+T^*T)X$ of point (2) we see that $(1+T^*T)^{\frac{1}{2}}X$
is dense in $A$. Using point (1) we conclude that $X$ is a core of
$T$.
\end{proof}
\begin{lem} \label{corprod1} Let $T_1,T_2\in A^\eta$ strongly commute
 and let $X\subset A$ be a dense subspace. Then the set
\[(1+(T_1T_2)^*(T_1T_2))(1+T_1^*T_1)^{-1}(1+T_2^*T_2)^{-1}X\] is dense
in $A$. In particular $(1+T_1^*T_1)^{-1}(1+T_2^*T_2)^{-1}X$ is a
core of $T_1T_2$.\end{lem}
\begin{proof}
Note that
\[(1+(T_1T_2)^*(T_1T_2))(1+T_1^*T_1)^{-1}(1+T_2^*T_2)^{-1}=(1+(T_1^*T_1)(T_2^*T_2))(1+T_1^*T_1)^{-1}(1+T_2^*T_2)^{-1}.\]
 We express the right hand side of the above
equation using $z$-transforms of $T_1$ and $T_2$:
\[(1-z_{|T_1|}^{2})(1-z_{|T_2|}^{2})+z_{|T_1|}^2z_{|T_2|}^2.\] Let
 $f:[0,1]\times[0,1]\rightarrow \mb{R}_+$  be the function defined
by
\[f(x_1,x_2)=(1-x_1^2)(1-x_2^2)+x_1^2x_2^2.\]
Note that $f(x_1,x_2)=0$ if and only if $x_1=1$ and $x_2=0$ or
$x_1=0$ and $x_2=1$. Let us  also define a function
$g:[0,1]\times[0,1]\rightarrow\mb{R}$  by the formula:
\[g(x_1,x_2)=(1-x_1^2)^{\frac{1}{2}}(1-x_2^2)^{\frac{1}{2}}.\]
We have the following implication:
\[(f(x_1,x_2)=0)\Rightarrow(g(x_1,x_2)=0).\]
Using Proposition 6.2 of \cite{WorNap} we get the following
inclusion:
\[\overline{(1+T_1^*T_1)^{-\frac{1}{2}}(1+T_2^*T_2)^{-\frac{1}{2}}X}^{\|\cdot\|}\subset
\overline{(1+(T_1^*T_1)(T_2^*T_2))(1+T_1^*T_1)^{-1}(1+T_2^*T_2)^{-1}X}^{\|\cdot\|}.\]
We end the proof by noting that
$A=\overline{(1+T_1^*T_1)^{-\frac{1}{2}}(1+T_2^*T_2)^{-\frac{1}{2}}X}^{\|\cdot\|}$.
\end{proof}
\begin{lem} \label{lemc5} Let $T_1,T_2\in A^\eta$ be a strongly commuting pair of operators
and let $Y\subset \D(T_2^*T_2)$  be such that $\displaystyle(1+T_2^*T_2)Y$ is
dense in $A$. Then the set
\[(1+(T_1T_2)^*(T_1T_2))(1+T_1^*T_1)^{-1}Y\] is dense in $A$.
\end{lem}
\begin{proof} The proof of this lemma follows from the previous one
with $X=(1+T_2^*T_2)Y$.
\end{proof}
\end{section}
\end{appendices}

\end{document}